\documentclass[10pt]{amsart}

\usepackage[margin=1in]{geometry}
\usepackage{amssymb}
\usepackage{tikz} 
\usepackage{hyperref}
\usepackage[all,cmtip]{xy}
\usepackage{booktabs}

\newtheorem{theorem}{Theorem}[section]

\newtheorem{proposition}[theorem]{Proposition}
\newtheorem{corollary}[theorem]{Corollary}

\theoremstyle{definition}

\newtheorem{example}[theorem]{Example}

\numberwithin{equation}{section}
\numberwithin{table}{section}
\numberwithin{figure}{section}

\DeclareMathOperator{\supp}{supp}

\def\SS{\mathfrak{S}}

\title{Dimensions of compositions modulo a prime}
\author{Jia Huang}
\address{Department of Mathematics and Statistics, University of Nebraska at Kearney, Kearney, NE 68849, USA}
\email{huangj2@unk.edu}

\begin{document}

\begin{abstract}
The (ordinary) representation theory of the symmetric group is fascinating and has rich connections to combinatorics, including the Frobenius correspondence to the self-dual graded Hopf algebra of symmetric functions. 
The $0$-Hecke algebra (of type $A$) is a deformation of the group algebra of the symmetric group, and its representation theory has an analogous correspondence to the dual graded Hopf algebras of quasisymmetric functions and noncommutative symmetric functions.
Macdonald used the hook length formula for the number of standard Young tableaux of a fixed shape to determine how many irreducible representations of the symmetric group have dimensions indivisible by a prime $p$.
In this paper, we study the dimensions of the projective indecomposable modules of the $0$-Hecke algebra modulo $p$; such a module is indexed by a composition and its dimension is given by a ribbon number, i.e., the cardinality of a descent class.
Applying a result of Dickson on the congruence of multinomial coefficients, we count how many ribbon numbers belong to each congruence class modulo $p$.
We also extend the result to other finite Coxeter groups.
\end{abstract}

\maketitle

\section{Introduction}

Given a finite group $G$ and a prime $p$, let $m_p(G)$ denote the number of (complex) irreducible representations of $G$ with dimension coprime to $p$.
For the \emph{symmetric group} $\SS_n$, which consists of all permutations of $[n]:=\{1,2,\ldots,n\}$, each irreducible representation is indexed by a \emph{partition} $\lambda$ of $n$, i.e., a decreasing sequence of positive integers $\lambda=(\lambda_1,\ldots,\lambda_\ell)$ with $|\lambda|:=\lambda_1+\cdots+\lambda_\ell=n$, and has dimension given by the number $f^\lambda$ of standard Young tableaux of shape $\lambda$.
Thus $f^\lambda$ is also known as the \emph{dimension} of the partition $\lambda$.
Using the $p$-core/quotient of a partition $\lambda$ and the well-known hook length formula for $f^\lambda$, Macdonald~\cite{Macdonald} showed that 
\[ 
m_p(\SS_n) = \prod_{j=0}^k \left[ \prod_{i=1}^\infty \frac{1}{(1-x^i)^{p^j}} \right]_{x^{n_j}},
\]
where $[f(x)]_{x^d}$ denotes the coefficient of $x^d$ in a power series $f(x)$ and 
\[ 
n=n_0+n_1 p + \cdots+ n_k p^k, \quad n_0, n_1, \ldots, n_k \in\{0,1,\ldots,p-1\}.
\]
In particular, if $n = p^{d_1} + \cdots +p^{d_k}$ is a sum of distinct powers of $p$, then $m_p(\SS_n) = p^{d_1+\cdots+d_k}$; this applies to all values of $n$ when $p=2$.

Extending Macdonald's result, Amrutha and T. Geetha~\cite{AG} obtained some results on $m_p(G)$ when $p$ is a power of $2$, and Khanna~\cite{Khanna} computed the number of partitions $\lambda$ of $n$ with $f^\lambda$ congruent to $1$ or $3$ modulo $4$ for certain values of $n$.
There have also been studies on partitions with odd dimensions~\cite{APS, GKNT} and more generally, on the divisibility of character values of $\SS_n$.
For example, Miller~\cite{Miller} conjectured that almost every character value of $\SS_n$ is congruent to $0$ modulo a prime $p$ as $n\to\infty$, Peluse~\cite{EvenChar} confirmed this conjecture for some small primes, and Peluse---Soundararajan~\cite{CharModPrime} established the conjecture for all prime moduli, although it remains open when the modulus is a power of a prime; see also Ganguly, Prasad, and Spallone~\cite{DivChar}.

On the other hand, there is a deformation of the group algebra of the symmetric group $\SS_n$ called the \emph{(type $A$) $0$-Hecke algebra} $H_n(0)$.
Similarly to the correspondence between the (complex) representation theory of $\SS_n$ and the self-dual graded Hopf algebra Sym of symmetric functions, the representation theory of $H_n(0)$, first studied by Norton~\cite{Norton}, admits a correspondence to the dual graded Hopf algebras QSym of quasisymmetric functions and $\mathbf{NSym}$ of noncommutative symmetric functions.
We briefly recall this correspondence below; see, e.g., Krob and Thibon~\cite{KrobThibon}.

Every irreducible $H_n(0)$-module $\mathbf{C}_\alpha$ is indexed by a \emph{composition} $\alpha$ of $n$, that is, a sequence $\alpha=(\alpha_1,\ldots,\alpha_\ell)$ of positive integers whose sum is $n$.
Every projective indecomposable $H_n(0)$-module is the projective cover $\mathbf{P}_\alpha$ of some $\mathbf{C}_\alpha$, so its \emph{top} (i.e., the quotient by its radical) is isomorphic to $\mathbf{C}_\alpha$. 
Each $\mathbf{C}_\alpha$ corresponds to the \emph{fundamental quasisymmetric function} $F_\alpha$, and each $\mathbf{P}_\alpha$ corresponds to the \emph{noncommutative ribbon Schur function} $\mathbf{s}_\alpha$; this gives two isomorphisms of graded Hopf algebras.
While $\mathbf{C}_\alpha$ is one dimensional, $\mathbf{P}_\alpha$ has a basis indexed by the \emph{descent class} $\{w\in \SS_n: D(w)=D(\alpha)\}$, where $D(w):=\{i\in[n-1]: w(i)>w(i+1)\}$ and $D(\alpha):=\{\alpha_1, \alpha_1+\alpha_2, \ldots, \alpha_1+\cdots+\alpha_{\ell-1}\}$; see also our earlier work~\cite{H0Tab} for a combinatorial realization of $\mathbf{P}_\alpha$  using standard tableaux of a ribbon shape corresponding to $\alpha$.
It follows that the dimension of $\mathbf{P}_\alpha$ is given by the \emph{(type $A$) ribbon number} 
\[ r_\alpha : = |\{ w\in \SS_n : D(w) = D(\alpha)\}|, \]
which can be viewed as the \emph{dimension} of the composition $\alpha$.
Note that $r_\alpha$ is also the \emph{flag $h$-vector} indexed by $D(\alpha)$ for the Boolean algebra of subsets of $[n-1]$.
Moreover, the descent classes of $\SS_n$ form a basis for the \emph{descent algebra}, which is an important subalgebra of the group algebra of $\SS_n$ introduced by Solomon~\cite{Solomon} and frequently studied in combinatorics and other areas.

Therefore, it is natural to study the number of compositions $\alpha$ of $n$ with $r_\alpha$ coprime to a given prime $p$, or more generally, the \emph{composition dimension $p$-vector} $c_p(n) := \left( c_{p,i}(n): i\in\mathbb{Z}_p \right)$, where
\[
c_{p,i}(n):= |\{\alpha\models n: r_\alpha \equiv i \pmod p\}| 
\quad \text{for all $i\in\mathbb{Z}_p$}. 
\]
In this paper, we use an expression of $r_\alpha$ as an alternating sum of multinomial coefficients and apply a theorem of Dickson~\cite{Dickson} on the congruence of multinomial coefficients modulo $p$ to determine $c_p(n)$.

Our result can be spelled out more explicitly for certain values of $n$, such as multiples of a power of $p$ and sums of distinct powers of $p$ (the latter includes all values of $n$ when $p=2$), but it becomes tedious for other values of $n$.
It would be nice to develop a different approach, even though it is not clear to us whether the results on $c_p(n)$ can be interpreted by operations on standard Young tableaux of ribbon shapes, the representation theory of the $0$-Hecke algebra $H_n(0)$, or the flag $h$-vector of the Boolean algebra of subsets of $[n-1]$.

It is possible to generalize our results to all finite Coxeter groups, for which the descent set and ribbon number are well defined.
For type $B$ and type $D$, we use the same method as in type $A$ to obtain similar results.
In particular, we show that every ribbon number in type $B$ and type $D$ is odd.
In contrast, the corresponding result in type $A$ is not as nice.
For example, we have $c_2(n)=(0,2^{n-1})$ if $n$ is a power of $2$, $c_2(n)=(2^{n-2},2^{n-2})$ if $n$ is a sum of two or three distinct powers of $2$, and $c_2(n)=2^{n-7}(35,29)$ if $n$ is a sum of four distinct powers of $2$.
For the exceptional types, we provide some data from computations in Sage.

This paper is structured as follows. 
First, we provide some preliminaries in Section~\ref{sec:prelim}.
Next, we give our results for type A in Section~\ref{sec:A}.
Then we extend our results to type $B$ and type $D$ in Section~\ref{sec:B} and Section~\ref{sec:D}, respectively.
Finally, we conclude the paper with a brief discussion on the exceptional types and some questions for future research in Section~\ref{sec:conclusion}.

\section{Preliminaries}\label{sec:prelim}

We first recall some basic definitions for Coxeter groups and their connections with combinatorics; see, e.g., Bj\"orner and Brenti~\cite{BjornerBrenti} for details.

Let $W$ be a group generated by a set $S$ with relations $(st)^{m_{st}} = 1$ for all $s, t\in S$, where $m_{st} = 1$ whenever $s=t$ and $m_{st} = m_{ts} \ge 2$ whenever $s\ne t$, or equivalently, $s^2=1$ for all $s\in S$ and $(sts\cdots)_{m_{st}} = (tst \cdots)_{m_{st}}$ whenever $s\ne t$, where $(aba\cdots)_m$ denotes the alternating product of $a$ and $b$ with length $m$.
The pair $(W,S)$ is called a \emph{Coxeter system} and $W$ is called a \emph{Coxeter group}.
We often label the elements of $S$ by nonnegative integers and identify each $s_i\in S$ with the index $i$.
The \emph{Coxeter diagram} of a Coxeter system $(W,S)$ is a graph whose vertices are the elements of $S$;
there is an edge between $s$ and $t$ whenever $m_{st}\ge 3$, and an edge is labeled with $m_{st}$ whenever $m_{st}\ge4$.
A Coxeter system is \emph{irreducible} if its Coxeter graph is connected.
Finite irreducible Coxeter systems are classified into types $A_n$, $B_n$, $D_n$, $E_6$, $E_7$, $E_8$, $F_4$, $H_3$, $H_4$, and $I_2(m)$.

For each $w\in W$, an expression $w=s_{i_1}\cdots s_{i_k}$ of $w$ as a product of elements in $S$ is \emph{reduced} if $k$ is as small as possible; the minimum value of $k$ is called the \emph{length} $\ell(w)$ of $w$.
The \emph{set of (right) descents} of $w$ is 
\[ 
D(w):= \{s\in S: \ell(ws)<\ell(w)\}. 
\]

As a $q$-deformation of the group algebra of a Coxeter system $(W,S)$, the \emph{Iwahori-Hecke algebra} $H_S(q)$ is generated by $\{T_s: s\in S\}$ with relations 
\begin{itemize}
\item
$(T_s+1)(T_s-q)=0$ for all $s\in S$, and
\item
$(T_sT_tT_s\cdots)_{m_{st}} = (T_tT_sT_t\cdots)_{m_{st}}$ for all distinct $s, t\in S$.
\end{itemize}
The algebra $H_S(q)$ has a (linear) basis $\{T_w: w\in W\}$, where $T_w := T_{i_1}\cdots T_{i_k}$ whenever $w=s_{i_1}\cdots s_{i_k}$ is a reduced expression.
While specializing $q=1$ recovers the group algebra of $W$, taking $q=0$ in the definition of $H_S(q)$ gives the \emph{$0$-Hecke algebra} $H_S(0)$.
Norton~\cite{Norton} worked out the representation theory of the $0$-Hecke algebra $H_S(0)$ of a finite Coxeter system $(W,S)$: 
Every irreducible $H_S(0)$-module $\mathbf{C}_I$ is one dimensional and indexed some $I\subseteq S$, and every projective indecomposable $H_S(0)$-module $\mathbf{P}_I$ is the projective cover of some $\mathbf{C}_I$ with a basis indexed by the \emph{descent class} $\{w\in W: D(W)=I\}$.
We are interested in the \emph{ribbon number}
\[ r_I^S := |\{ w\in W: D(w)=I\}| = \dim(\mathbf{P}_I). \]

Every subset $I$ of $S$ generates a \emph{parabolic subgroup} $W_I$ of $W$.
A (left) coset of $W_I$ has a unique representative of minimum length, which is the element $w$ in this coset with $D(w)\subseteq S\setminus I$.
Thus 
\[ \left|W/W_{S\setminus I}\right| = \left|\left\{ w\in W: D(w)\subseteq I \right\}\right| = \sum_{J\subseteq I} r_J^S. \]
It follows from inclusion-exclusion that
\[ r_I^S = \sum_{J\subseteq I} (-1)^{|I|-|J|} \left|W/W_{S\setminus I}\right|. \]

A finite Coxeter system $(W,S)$ has a longest element $w_0\in W$, which satisfies $\ell(w_0w)=\ell(w_0)-\ell(w)$ for all $w\in W$.
This implies that $D(w_0w)=S\setminus D(w)$ since for each $s\in S$, we have 
\[ \ell(w_0ws) = \ell(w_0)-\ell(ws)>\ell(w_0)-\ell(w)=\ell(w_0w) \Longleftrightarrow \ell(ws) < \ell(w). \]
Therefore, there is a symmetry $r_I^S = r_{S\setminus I}^S$ among the ribbon numbers.

An important example of the Coxeter groups is the \emph{symmetric group} $\SS_n$, which consists of all permutations of $[n]:=\{1,2,\ldots,n\}$.
It is generated by $s_1, \ldots, s_{n-1}$, where $s_i:=(i,i+1)$ is the \emph{adjacent transposition} that swaps $i$ and $i+1$ for $i=1,2,\ldots,n-1$, with the following relations:
\[\begin{cases}
s_i^2 = 1, \ 1\le i\le n-1, \\
s_is_{i+1}s_i = s_{i+1}s_is_{i+1},\ 1\le i\le n-2,\\
s_is_j=s_js_i,\ 1\le i,j\le n-1,\ |i-j|>1.
\end{cases}\]
Thus $\SS_n$ is the Coxeter group of type $A_n$, whose Coxeter diagram is below.
\[
\xymatrix @C=12pt{
s_1 \ar@{-}[r] & s_2 \ar@{-}[r] & \cdots \ar@{-}[r] & s_{n-2} \ar@{-}[r] & s_{n-1}
}
\]
For each $w\in\SS_n$, the length $\ell(w)$ coincides with the number of inversion pairs
\[ \mathrm{inv}(w) := |\{(i,j): 1\le i<j\le n,\ w(i) > w(j) \}|, \]
and identifying $s_i$ with $i$, we have 
\[ D(w) = \{ i\in[n-1]: w(i) > w(i+1) \} . \]

The (type $A$) $0$-Hecke algebra $H_n(0)$ is the monoid algebra of the monoid generated by $\pi_1, \ldots, \pi_{n-1}$ (note that $\pi_i$ is not $T_i$ but rather $T_i+1$ when $q=0$) with relations
\[\begin{cases}
\pi_i^2 = \pi_i, \ 1\le i\le n-1, \\
\pi_i\pi_{i+1}\pi_i = \pi_{i+1}\pi_i\pi_{i+1},\ 1\le i\le n-2,\\
\pi_i\pi_j=\pi_j\pi_i,\ 1\le i,j\le n-1,\ |i-j|>1.
\end{cases}\]

To describe the representation theory of $H_n(0)$, recall that a \emph{composition} is a sequence $\alpha = (\alpha_1, \ldots, \alpha_\ell)$ of positive integers.
The \emph{size} of $\alpha$ is $|\alpha| := \alpha_1+\cdots+\alpha_\ell$ and the \emph{parts} of $\alpha$ are $\alpha_1,\ldots,\alpha_\ell$.
If $|\alpha|=n$ then we say $\alpha$ is a \emph{composition of $n$} and write $\alpha \models n$.
Define $\sigma_i := \alpha_1 + \cdots + \alpha_i$ for $i=0,1,\ldots,\ell$.
It is clear that $\sigma_0=0$, $\sigma_\ell=n$, and $D(\alpha) := \{\sigma_1,\ldots,\sigma_{\ell-1}\}$ is a subset of $[n-1]$.
Thus a composition $\alpha\models n$ can be encoded in a binary string $a_1\cdots a_n$ whose $j$th digit is one if and only if $j\in D(\alpha)\cup\{n\}$.
It follows that there are exactly $2^{n-1}$ compositions of $n$.
The \emph{length} of $\alpha$ is $\ell(\alpha) := \ell = |D(\alpha)|+1$.

Each irreducible $H_n(0)$-module $\mathbf{C}_\alpha$ is indexed by a composition $\alpha$ of $n$ and has dimension one.
Each projective indecomposable $H_n(0)$-module $\mathbf{P}_\alpha$ is the projective cover of some $\mathbf{C}_\alpha$ and has a basis indexed by permutations of $[n]$ with descent set equal to $D(\alpha)$, so its dimension is the \emph{(type $A$) ribbon number}
\[ 
r_\alpha := | \{ w\in\SS_n: D(w) = D(\alpha) \} |.
\]
We call $r_\alpha$ the \emph{dimension} of the composition $\alpha$.

To compute $r_\alpha$, we use the following \emph{multinomial coefficient}, where $m, m_1, \ldots, m_k$ are nonnegative integers.
\[ \binom{m}{m_1,\ldots,m_k} :=
\begin{cases}
\frac{m!}{m_1!\cdots m_\ell!} & \text{if } m_1+\cdots+m_k=m, \\
0 & \text{otherwise.} 
\end{cases} \]
Given a composition $\beta=(\beta_1, \ldots, \beta_\ell)$ of $n$, a permutation $w\in\SS_n$ satisfies $D(w) \subseteq D(\beta)$ if and only if $w(j) < w(j+1)$ for all $j\in[n-1]\setminus D(\beta)$.
Thus 
\[ | \{ w\in \SS_n: D(w) \subseteq D(\beta) \}| = \binom{n}{\beta} := \binom{n}{\beta_1, \ldots, \beta_\ell}. \]
Applying inclusion-exclusion to this yields a formula
\begin{equation}\label{eq:ribbon}
r_\alpha = \sum_{ \beta \preccurlyeq \alpha } (-1)^{\ell(\alpha) - \ell(\beta)} \binom{n}{\beta} 
\end{equation}
for every composition $\alpha$ of $n$, where $\beta \preccurlyeq \alpha$ means that $\beta$ is a composition of $n$ with $D(\beta)\subseteq D(\alpha)$, or in other words, $\beta$ is a composition of $n$ refined by $\alpha$.
There is also an equivalent formula for $r_\alpha$ given below, where $1/k! := 0$ if $k<0$:
\[ r_\alpha = n! \det \left( \frac1{(\sigma_j-\sigma_{i-1})!} \right)_{i,j=1}^\ell \]

We will mainly use the formula~\eqref{eq:ribbon} for $r_\alpha$ to study the \emph{type $A$ composition dimension $p$-vector} $c_p(n) := \left(c_{p,i}(n): i\in\mathbb{Z}_p\right)$, where
\[ 
c_{p,i}(n) := |\{ \alpha\models n: r_\alpha \equiv i \pmod p\}.
\]
We can write a positive integer $n$ in base $p$ as $n=n_0+n_1p+\cdots+n_kp^k$, where $n_0, n_1, \ldots, n_k\in\{0,1,\ldots,p-1\}$ and $n_k>0$; this gives a vector $[n]_p:=(n_0,n_1,\ldots,n_k)$.
A well-known theorem of Lucas determined the residue of a binomial coefficient modulo a prime $p$, and Dickson~\cite[p.~76]{Dickson} generalized it to the multinomial coefficient.
\begin{theorem}[Dickson]\label{thm:Dickson}
Let $n$ be a positive integer $n$ with $[n]_p=(n_0,n_1,\ldots,n_k)$.
Given nonnegative integers $\beta_1,\ldots,\beta_\ell$ whose sum is $n$, we write $[\beta_i]_p=(\beta_{i0},\beta_{i1},\ldots,\beta_{ik})$ for $i=1,\ldots,\ell$ by abuse of notation (adding trailing zeros if necessary).
Then
\[ \binom{n}{\beta_1,\ldots,\beta_\ell} \equiv \prod_{j=0}^k \binom{n_j}{\beta_{1j}, \ldots, \beta_{\ell j}} \pmod p. \]
\end{theorem}

It follows from Theorem~\ref{thm:Dickson} that $\binom{n}{\beta_1,\ldots,\beta_\ell}\equiv0\pmod p$ unless $\beta_{1j}+\cdots+\beta_{\ell j} = n_j$ for all $j=0,1,\ldots,k$, i.e., $([\beta_1]_p, \ldots, [\beta_\ell]_p)$ is a \emph{vector composition} of $[n]_p$ (cf. Andrews~\cite{Andrews}).

Also recall that a \emph{poset} is a set $P$ with a partial order on $P$.
A \emph{chain} in $P$ is a subset of $P$ whose elements are pairwise comparable.
The compositions of $[n]$ form a poset under $\preccurlyeq$, which is isomorphic to the poset of subsets of $[n-1]$ under inclusion via the bijection $\alpha\mapsto D(\alpha)$; this gives two incarnations of the \emph{finite Boolean algebra}.

\section{Type A}\label{sec:A}

Let $p$ be a prime number. 
In this section we determine the type $A$ composition dimension $p$-vector $c_p(n):=\left(c_{p,i}(n):i\in\mathbb{Z}_p \right)$, where $c_{p,i}(n)$ is the number of compositions $\alpha\models n$ satisfying $r_\alpha\equiv i\pmod p$ for all $i\in\mathbb{Z}_p$.

\begin{theorem}\label{thm:A}
Let $p$ be a prime and $n\ge2$ an integer with $[n]_p=(n_0,n_1,\ldots,n_k)$.
Define
\[ P := \left\{ 
b_0+b_1 p+\cdots+b_k p^k: 0\le b_j\le n_j,\ j=0,1,\ldots,k 
\right \}
\setminus \{0,n\}, \]
which is a subset of $[n-1]$ with $|P|=\prod_{j=0}^k (n_j+1) -2$.
For each $T\subseteq P$, define 
\[ 
r(T) := \sum_{ \substack{ \beta\models n,\ D(\beta) \subseteq T\\
\beta_{1j}+\cdots+\beta_{\ell j} = n_j,\ \forall j}}
(-1)^{|T|-|D(\beta)|} 
\prod_{j=0}^k \binom{n_j}{\beta_{1j}, \ldots, \beta_{\ell j}}.
\]
Here $\beta=(\beta_1,\ldots,\beta_\ell)$ is a composition of $n$ with $[\beta_i]_p = (\beta_{i0}, \beta_{i1}, \ldots, \beta_{id})$ for $i=1,\ldots,\ell$.
Then
\[ c_{p,i}(n) = 
\begin{cases}
2^{n+1-\prod_{j=0}^k (n_j+1)} \left| \{ T\subseteq P: r(T)\equiv i\pmod p \} \right| & \text{if $p=2$ or $i=0$ or $n_0=\cdots=n_{k-1}=p-1$}, \\
2^{n-\prod_{j=0}^k (n_j+1)} \left| \{ T\subseteq P: r(T)\equiv \pm i\pmod p \} \right| & \text{otherwise}.
\end{cases} \]
\end{theorem}

\begin{proof}
Let $\alpha$ be an arbitrary composition of $n$, whose corresponding binary string is $a=a_1\cdots a_n$ with $a_n=1$.
We use the formula~\eqref{eq:ribbon} for the ribbon number $r_\alpha$ to reduce it modulo $p$.
We have $\beta \preccurlyeq \alpha$ if and only if $a_r=1$ for all $r$ in
\[ 
D(\beta) = \left\{ \sum_{i=1}^s \sum_{j=0}^k \beta_{ij}p^j: s=1,\ldots\ell-1 \right\}.
\]
If $\binom{n}{\beta}\not\equiv 0\pmod p$ then $\beta_{1j}+\cdots+\beta_{\ell j} = n_j$ for all $j=0,1,\ldots,k$ by Theorem~\ref{thm:Dickson}, 
and this implies $D(\beta)\subseteq P$.
Thus to find which compositions $\beta$ with $\binom{n}{\beta}\not\equiv 0\pmod p$ are refined by $\alpha$, it suffices to look at the substring $\hat a := (a_r: r\in P)$ of $a$.
It is easy to see that $P$ is a subset of $[n-1]$ with $|P|=(n_0+1)\cdots(n_k+1)-2$.

Let $b$ be a fixed binary string indexed by $P$ with $\supp(b) := \{r\in P: b_r=1\}=T$. 
If $\hat a=b$, then  
\[
r_\alpha \equiv \sum_{ \substack{ \beta\models n \\ 
D(\beta) \subseteq T } } 
(-1)^{\ell(\alpha)-\ell(\beta)} \prod_{j=0}^k \binom{n_j}{\beta_{1j}, \ldots, \beta_{\ell j}}
\equiv r(T) \pmod p,
\]
and we have exactly $2^{n-|P|-1}$ possibilities for $\alpha$, half of which have even lengths by toggling $a_j$ for some $j\in[n-1]\setminus P$ unless $P=[n-1]$, i.e., $n_0=n_1=\cdots=n_{k-1}=p-1$.
The result follows.
\end{proof}

We derive some consequences of Theorem~\ref{thm:A} below.

\begin{corollary}\label{cor:A}
Let $p$ be a prime and $n\ge2$ an integer with $[n]_p=(n_0,\ldots,n_k)$.
For all $i\in\mathbb{Z}_p$, we have $c_{p,i}(n)=c_{p,-i}(n)$ unless $n_0=\cdots=n_{k-1}=p-1$ and 
\[\begin{cases}
2^{n+2-(n_0+1)\cdots(n_k+1)} \text{ divides } c_{p,i}(n) , & \text{if $p=2$ or $i=0$ or $n_0=\cdots=n_{k-1}=p-1$}, \\
2^{n+1-(n_0+1)\cdots(n_k+1)} \text{ divides } c_{p,i}(n) , & \text{otherwise}.
\end{cases} \]
\end{corollary}

\begin{proof}
Theorem~\ref{thm:A} immediately implies that $c_{p,i}(n)=c_{p,-i}(n)$ for all $i\in\mathbb{Z}_p$ unless $n_0=\cdots=n_{k-1}=p-1$.
The symmetry $r_I^S=r_{S\setminus I}^S$ for the ribbon numbers of a finite Coxeter system $(W,S)$ mentioned in Section~\ref{sec:prelim} implies $r(T) \equiv r(P\setminus T) \pmod p$ for all $T\subseteq P$.
Thus $c_{p,i}(n)$ is divisible by the desired power of $2$.
\end{proof}

Next, we specialize Theorem~\ref{thm:A} to the case when $n$ is a multiple of a prime power.

\begin{corollary}\label{cor:PowerA}
Suppose $n=mp^d$, where $p$ is a prime, $m\in\{1,\ldots,p-1\}$, and $d\ge0$ is an integer.
Then
\[ c_{p,i}(n) =
\begin{cases}
2^{n-m} |\{\gamma\models m: r_\gamma \equiv i \pmod p\}| & \text{if $i=0$ or $p=2$ or $d=0$}, \\
2^{n-m-1} |\{\gamma\models m: r_\gamma \equiv \pm i \pmod p\}| & \text{otherwise}.
\end{cases}\]
\end{corollary}

\begin{proof}
By Theorem~\ref{thm:A}, we have $[n]_p=(0,\ldots,0,m)$, $P=\{jp^d: j=1,2,\ldots,m-1\}$, and each $T\subseteq P$ corresponds to a composition $\gamma\models m$ with $D(\gamma)=\{t/p^d: t\in T\}$ such that
\begin{align*}
r(T) & = \sum_{ \substack{ \beta\models n \\ D(\beta) \subseteq T } } 
(-1)^{|T|-|D(\beta)|} \prod_{j=0}^d \binom{n_j}{\beta_{1j}, \ldots, \beta_{\ell j}} \\
&= \sum_{ \delta\preccurlyeq \gamma } 
(-1)^{\ell(\gamma)-\ell (\delta)} \binom{m}{\delta} = r_\gamma, \\
\end{align*}
where $\delta$ is obtained from $\beta$ by dividing each part by $p^d$.
The result follows.
\end{proof}

\begin{example}
Corollary~\ref{cor:PowerA} becomes trivial when $d=0$. 
Assume $d>0$ below.
For $m=1,2,3,4$, we compute $r_\gamma$ for all $\gamma\models m$:
\begin{itemize}
\item
$r_{(1)}=1$, $r_{(1,1)}=r_{(2)}=1$, $r_{(1,1,1)}=r_{(3)}=1$, $r_{(1,2)} = r_{(2,1)} = 2$,
\item
$r_{(1,1,1,1)}=r_{(4)}=1$, $r_{(1,1,2)}=r_{(2,1,1)}=r_{(1,3)}=r_{(3,1)}=3$, $r_{(1,2,1)}=r_{(2,2)}=5$.
\end{itemize}
Thus we have the following by Corollary~\ref{cor:PowerA}.  
\begin{itemize}
\item
If $n=p^d$ then $c_p(n)=(0,2^{n-1})$ when $p=2$ and $c_p(n)=(0,2^{n-2},0,\ldots,0,2^{n-2})$ when $p>2$.
\item
If $n=2p^d$ and $p>2$ then $c_{p,\pm1}(n) = 2^{n-2}$ and $c_{p,i}(n)=0$ for all $i\not \equiv \pm 1\pmod p$.
\item
If $n=3p^d$ and $p>3$ then $c_{p,\pm 1}(n) = c_{p,\pm3}(n) = 2^{n-3}$ and $c_{p,i}(n)=0$ for all $i\not\equiv \pm1,\pm3 \pmod p$.
\item
If $n=4p^d$ then $c_p(n)=2^{n-4}(2,1,2,2,1)$ when $p=5$ and $c_{p,\pm1}(n) = c_{p,\pm5}(n) = 2^{n-4}$, $c_{p,\pm3}(n) =2^{n-3}$, and $c_{p,i}(n)=0$ for all $i\not\equiv \pm1, \pm3, \pm5 \pmod p$ when $p>5$.
\end{itemize}
\end{example}

Next, we consider the case when $n$ is a sum of distinct powers of a prime $p$; this applies to all values of $n$ when $p=2$.

\begin{corollary}\label{cor:PowersA}
Let $p$ be a prime and $n=p^{d_1}+\cdots+p^{d_k}$, where $0\le d_1<\cdots<d_k$ and $k>1$.
Define 
\[ 
P:= \left\{ \overline U: \emptyset\ne U\subsetneqq \{p^{d_1}, \ldots, p^{d_k} \} \right\}, 
\quad \text{partially ordered by } \preccurlyeq,
\] 
where $\overline U$ denotes the sum of all elements of $U$ and $\overline U \preccurlyeq \overline V$ in the poset $P$ if and only if $U\subseteq V$.
Then
\[ c_{p,i}(n) = 
\begin{cases}
2^{n-2^k+1}|\{ T\subseteq P: \chi(T) \equiv i \pmod p \}|, & \text{if  $p=2$ or $i=0$} \\
2^{n-2^k}|\{ T\subseteq P: \chi(T) \equiv \pm i \pmod p \}|, & \text{if $p>2$ and $i=1,\ldots,p-1$}.
\end{cases} \]
Here $\chi(T)$ is the number of chains of even cardinality minus the number of chains of odd cardinality in $T$. 
\end{corollary}

\begin{proof}
By Theorem~\ref{thm:A}, we have $P= \left\{ \overline U: \emptyset\ne U\subsetneqq \{p^{d_1}, \ldots, p^{d_k} \} \right\} $ with $|P|=2^k-2$, and for every $T\subseteq P$, 
\[
r(T) = (-1)^{|T|} \chi(T) 
\]
since a composition $\beta\models n$ satisfies $D(\beta)\subseteq T$ if and only if $D(\beta)$ gives a chain of cardinality $|D(\beta)|$ in $T$ and the multinomial coefficients in the definition of $r(T)$ are all ones for $n=p^{d_1}+\cdots+p^{d_k}$.
The result then follows from Theorem~\ref{thm:A} (the case $n_0=\cdots=n_{k-1}=p-1$ does not occur since $k>1$).
\end{proof}

We give an example when $n$ is a sum of two distinct powers of a prime $p$.

\begin{example}
Let $p$ be a prime and $n=u+v$, where $u$ and $v$ are distinct nonnegative powers of $p$.
By Corollary~\ref{cor:PowersA}, we have $P=\{u,v\}$, $\chi(\emptyset)=1-0=1$, $\chi(\{u\})=\chi(\{v\})=1-1=0$, $\chi(\{u,v\})=1-2=-1$, and thus the following holds.
\begin{itemize}
\item
If $p=2$ then $c_{p,0}(n)=c_{p,1}(n)=2^{n-2}$.
\item
If $p>2$ then $c_{p,0}(n)=2^{n-2}$, $c_{p,\pm1}(n)=2^{n-3}$, and $c_{p,i}(n)=0$ for all $i\not\equiv 0, \pm1 \pmod p$.
\end{itemize}
We also provide a direct proof here to help illustrate our method. Theorem~\ref{thm:Dickson} implies that 
\[ \binom{n}{\beta} \equiv 
\begin{cases}
1 \pmod p & \text{if } \beta\in\{ (n), (u,v), (v,u)\}; \\
0 & \text{otherwise}
\end{cases} \]
for every composition $\beta\models n$.
Let $\alpha$ be a composition of $n$ corresponding to a binary string $a_1\cdots a_n$.
Then
\[ r_\alpha \equiv 
\begin{cases}
(-1)^{\ell(\alpha)} \pmod p & \text{if } (u,v) \preccurlyeq \alpha \text{ and } (v,u) \preccurlyeq \alpha, \text{ i.e., } a_u = a_v = 1, \\
(-1)^{\ell(\alpha)-1} \pmod p & \text{if } (u,v) \not\preccurlyeq \alpha \text{ and } (v,u) \not\preccurlyeq \alpha, \text{ i.e., } a_v = a_u = 0, \\
0 & \text{otherwise}.
\end{cases} \]
There are exactly $2^{n-3}$ possibilities for $\alpha$ in either the first or the second case.
There are $2^{n-2}$ possibilities for $\alpha$ in the last case, half with even lengths in each case by the bijection toggling $a_j$ for some $j\in[n-1]\setminus P$.
The result follows.
\end{example}

We have another example when $n$ is the sum of three distinct powers of a prime $p$.

\begin{example}
Suppose $n=u+v+w$, where $u,v,w$ are distinct nonnegative powers of $p$.
By Corollary~\ref{cor:PowersA}, we need to calculate $\chi(T)$ for each $T\subseteq P:=\{u,v,w,u+v,u+w,v+w\}$.
We distinguish some cases below.
\begin{itemize}
\item
If $|T|=0$ then $\chi(T)=1-0=1$; the number of possibilities for $T$ is $1$.
\item
If $|T|=1$ then $\chi(T)=1-1=0$; the number of possibilities for $T$ is $6$.
\item
If $T$ consists of two comparable elements, then $\chi(T)=2-2=0$; the number of possibilities for $T$ is $3\cdot 2 = 6$.
\item
If $T$ consists of two incomparable elements, then $\chi(T)=1-2=-1$; the number of possibilities for $T$ is $\binom{6}{2}-6=9$.
\item
If $T$ is of the form $\{u,v,u+v\}$ or $\{u,u+v,u+w\}$, then $\chi(T)=3-3=0$; the number of possibilities for $T$ is $3+3=6$.
\item
If $T$ is for the form $\{u,v,u+w\}$ or $\{u,u+v,v+w\}$, then $\chi(T)=2-3=-1$;
the number of possibilities for $T$ is $3\cdot 2+3\cdot 2=12$.
\item
If $T$ is of the form $\{u,v,w\}$ or $\{u+v,u+w,v+w\}$, then $\chi(T)=1-3=-2$; the number of possibilities for $T$ is $2$.
\item
If $T$ is of the form $\{u,v,w,u+v\}$, $\{u,u+v,u+w,v+w\}$, or $\{u,v,u+w,v+w\}$, then $\chi(T)=3-4=-1$;
the number of possibilities for $T$ is $3+3+3=9$.
\item
If $T$ is of the form $\{u,v,u+v,u+w\}$, then $\chi(T)=4-4=0$; the number of possibilities for $T$ is $3\cdot 2=6$.
\item
If $|T|=5$, then $\chi(T)=5-5=0$; the number of possibilities for $T$ is $6$.
\item
If $|T|=6$, then $\chi(T)=7-6=1$; the number of possibilities for $T$ is $1$.
\end{itemize}
Thus by Corollary~\ref{cor:PowersA}, we have 
\begin{itemize}
\item
$c_{2,0}(n)=2^{n-7}(6+6+6+2+6+6)=32\cdot 2^{n-7}=2^{n-2}$, 
\item
$c_{2,1}(n)=2^{n-7}(1+9+12+9+1)=32\cdot 2^{n-7}=2^{n-2}$,
\item
$c_{p,0}(n)=2^{n-7}(6+6+6+6+6)=30\cdot 2^{n-7}$ if $p>2$,
\item
$c_{p,\pm1}(n)=2^{n-8}(1+9+12+2+9+1)=17\cdot 2^{n-7}$ if $p=3$,
\item
$c_{p,\pm1}(n)=2^{n-8}(1+9+12+9+1)=16\cdot 2^{n-7}$ if $p>3$,
\item
$c_{p,\pm2}(n)=2^{n-8}\cdot2=2^{n-7}$ if $p>3$. \qedhere
\end{itemize}
\end{example}

The next example settles the case when $n$ is a sum of four distinct powers of $p$.

\begin{example}
Computations in Sage based on Corollary~\ref{cor:PowersA} show that if $n$ is the sum of four distinct nonnegative powers of $p$ then 
\[ c_p(n)=
\begin{cases}
2^{n-15}(8960, 7424)=2^{n-7}(35,29), & \text{if } p = 2, \\
2^{n-15}(7766, 4309, 4309), & \text{if } p = 3, \\
2^{n-15}(7606, 3636, 753, 753, 3636), & \text{if } p = 5, \\
2^{n-15}(7604, 3630, 673, 87, 87, 673, 3630), & \text{if } p=7,\\
2^{n-15}(7604, 3630, 672, 81, 6, 1, 0,\ldots,0, 1, 6, 81, 672, 3630), & \text{if } p\ge 11.
\end{cases} \]
\end{example}

Next, we study the case when $n$ is two times a power of an odd prime $p$ plus a different power of $p$.

\begin{corollary}
Let $p$ be an odd prime and $n=2p^d + p^e$, where $d$ and $e$ are distinct nonnegative integers.
\begin{itemize}
\item
If $p=3$ and $n=5$ then $c_{p,0}(n)=6$, $c_p^1(n)=8$, and $c_p^2(n)=2$.
\item
If $p=3$ and $n>5$ then $c_{p,0}(n)=6\cdot 2^{n-5}$ and $c_{p,\pm1}(n)=5\cdot 2^{n-5}$.
\item
If $p>3$ then $c_{p,0}(n) = 6\cdot 2^{n-5}$, $c_{p,\pm1}(n) = 4\cdot 2^{n-5}$, $c_{p,\pm2}(n) = 2^{n-5}$, and $c_{p,i}(n) = 0$ for all $i\not\equiv 0,\pm1,\pm2 \pmod p$.
\end{itemize}
\end{corollary}

\begin{proof}
By Theorem~\ref{thm:A}, we have $P=\{p^d, p^e, 2p^d, p^d+p^e\}$, and for each $T\subseteq P$, we can calculate $r(T)$ based on which $\beta\models n$ satisfies $D(\beta)\subseteq T$ in the definition of $r(T)$.
\begin{itemize}
\item
If $T=\{p^d,p^e,2p^d,p^d+p^e\}$ then $r(T)=1-1-2-2-1+2+2+2=1$ since the possibilities for $\beta$ form a poset with the following Hasse diagram.
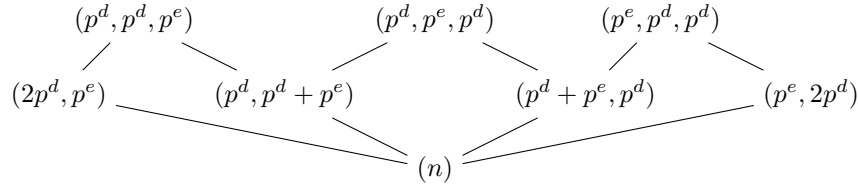
\begin{center}
\begin{tikzpicture}
\node (0) at (0,0) {$(n)$};
\node (1) at (-5,1) {$(2p^d,p^e)$};
\node (2) at (-2,1) {$(p^d,p^d+p^e)$};
\node (3) at (2,1) {$(p^d+p^e,p^d)$};
\node (4) at (5,1) {$(p^e,2p^d)$};
\node (12) at (-4,2) {$(p^d,p^d,p^e)$};
\node (23) at (0,2) {$(p^d,p^e,p^d)$};
\node (34) at (3,2) {$(p^e,p^d,p^d)$};

\draw (0) -- (1); \draw (0) -- (2); \draw (0) -- (3); \draw (0) -- (4);
\draw (12) -- (1); \draw (12) -- (2); 
\draw (23) -- (2); \draw (23) -- (3);
\draw (34) -- (3); \draw (34) -- (4);
\end{tikzpicture}
\end{center}
\item
If $T=\{p^d,p^e,2p^d\}$ then $\beta\in\{(n),(p^d,p^d+p^e),(p^e,2p^d),(2p^d,p^e),(p^d,p^d,p^e)\}$ and thus  $r(T)=(-1)^3(1-2-1-1+2)=1$.
\item
If $T=\{p^d,p^e,p^d+p^e\}$ then $\beta\in\{(n),(p^d,p^d+p^e),(p^e,2p^d),(p^d+p^e,p^d),(p^d,p^e,p^d)(p^e,p^d,p^d)\}$ and thus $r(T)=-1+2+1+2-2-2=0$.
\item
If $T=\{p^d,2p^d,p^d+p^e\}$ then $\beta\in\{(n),(p^d,p^d+p^e),(2p^d,p^e),(p^d+p^e,p^d),(p^d,p^d,p^e),(p^d,p^e,p^d)\}$ and thus $r(T)=-1+2+1+2-2-2=0$.
\item
If $T=\{p^e,2p^d,p^d+p^e\}$ then $\beta\in\{(n),(p^e,2p^d),(2p^d,p^e),(p^d+p^e,p^d),(p^e,p^d,p^d)\}$ and thus $r(T)=-1+1+1+2-2=1$.
\item
If $T=\{p^d,p^e\}$ then $\beta\in\{(n),(p^d,p^d+p^e),(p^e,2p^d)\}$ and $r(T)=1-2-1=-2$.
\item
If $T=\{p^d,2p^d\}$ then $\beta\in\{(n),(p^d,p^d+p^e),(2p^d,p^e), (p^d,p^d,p^e)\}$ and $r(T)=1-2-1+2=0$.
\item
If $T=\{p^d,p^d+p^e\}$ then $\beta\in\{(n),(p^d,p^d+p^e),(p^d+p^e,p^d), (p^d,p^e,p^d)\}$ and $r(T)=1-2-2+2=-1$.
\item
If $T=\{p^e,2p^d\}$ then $\beta\in\{(n),(p^e,2p^d),(2p^d,p^e)\}$ and $r(T)=1-1-1=-1$.
\item
If $T=\{p^e,p^d+p^e\}$ then $\beta\in\{(n),(p^e,2p^d),(p^d+p^e,p^d),(p^e,p^d,p^d)\}$ and $r(T)=1-1-2+2=0$.
\item
If $T=\{2p^d,p^d+p^e\}$ then $\beta\in\{(n),(2p^d,p^e),(p^d+p^e,p^d)\}$ and $r(T)=1-1-2=-2$.
\item
If $T=\{p^d\}$ then $\beta=\in\{(n),(p^d,p^d+p^e)\}$ and $r(T)=-1+2=1$.
\item
If $T=\{p^e\}$ then $\beta\in\{(n),(p^e,2p^d)\}$ and $r(T)=-1+1=0$.
\item
If $T=\{2p^d\}$ then $\beta=\{(n),(2p^d,p^e)\}$ and $r(T)=-1+1=0$.
\item
If $T=\{p^d+p^e\}$ then $\beta\in\{(n),(p^d+p^e,p^d)$ and $r(T)=-1+2=1$.
\item
If $T=\emptyset$ then $\beta=(n)$ and $r(T)=1$.
\end{itemize}
Thus by Theorem~\ref{thm:A}, we have 
\begin{itemize}
\item
$c_{3,0}(5)=6$, $c_{3,1}(5)=8$, $c_{3,2}(5)=2$,
$c_{3,0}(n) = 6\cdot 2^{n-5}$ and $c_{3,\pm1}(n) = 10\cdot 2^{n-6}$ for $n>5$,
\item
$c_{p,0}(n) = 6\cdot 2^{n-5}$, $c_{p,\pm1}(n) = 8\cdot 2^{n-6}$, $c_{p,\pm2}(n)=2\cdot 2^{n-6}$, and $c_{p,i}(n)=0$ for all $i\not\equiv 0,\pm1,\pm2 \pmod p$ if $p>3$. \qedhere
\end{itemize}
\end{proof}

The case $n=2p^d+2p^e$ with $d\ne e$ will again be tedious.
Computations in Sage based on the definition of $c_p(n)$ show that $c_3(20)=c_3(220_3)=2^{12}(42,43,43)$. 

We compute $c_p(n)$ based on its definition for small values of $n$ and $p$ in Sage and give our data in Table~\ref{tab:A}, which agree with the results in this section; 
note that the power of $2$ given by Corollary~\ref{cor:A} may or may not be the highest in $c_{p,i}(n)$.

\begin{table}[h]
\begin{center}
\tiny
\renewcommand{\arraystretch}{1.2}
\begin{tabular}{c|ccccc} 
\toprule
$n$ & $p=2$ & $p=3$ & $p=5$ & $p=7$ & $p=11$ \\ 
\midrule
$2$ & $2(0, 1)$ & $2(0, 1, 0)$ & $2(0, 1, 0, 0, 0)$ & $2(0, 1, 0, 0, 0, 0, 0)$ & $2(0, 1, 0, 0, 0, 0, 0, 0, 0, 0, 0)$ \\
$3$ & $2(1, 1)$ & $2(0, 1, 1)$ & $2(0, 1, 1, 0, 0)$ & $2(0, 1, 1, 0, 0, 0, 0)$ & $2(0, 1, 1, 0, 0, 0, 0, 0, 0, 0, 0)$ \\
$4$ & $2^3(0, 1)$ & $2(2, 1, 1)$ & $2(1, 1, 0, 2, 0)$ & $2(0, 1, 0, 2, 0, 1, 0)$ & $2(0, 1, 0, 2, 0, 1, 0, 0, 0, 0, 0)$ \\
$5$ & $2^3(1, 1)$ & $2(3, 4, 1)$ & $2^3(0, 1, 0, 0, 1)$ & $2(0, 1, 3, 0, 3, 0, 1)$ & $2(1, 1, 0, 0, 2, 1, 1, 0, 0, 2, 0)$ \\
$6$ & $2^4(1, 1)$ & $2^4(0, 1, 1)$ & $2^3(2, 1, 0, 0, 1)$ & $2(4, 1, 0, 2, 0, 9, 0)$ & $2(0, 1, 2, 2, 2, 2, 1, 2, 2, 0, 2)$ \\
$7$ & $2^5(1, 1)$ & $2^2(6, 5, 5)$ & $2^2(6, 4, 1, 1, 4)$ & $2^5(0, 1, 0, 0, 0, 0, 1)$ & $2(5, 9, 1, 0, 3, 3, 4, 1, 1, 5, 0)$ \\
$8$ & $2^7(0, 1)$ & $2(21, 17, 26)$ & $2(22, 9, 12, 12, 9)$ & $2^5(2, 1, 0, 0, 0, 0, 1)$ & $2(6, 7, 5, 5, 8, 7, 9, 6, 5, 1, 5)$ \\
$9$ & $2^7(1, 1)$ & $2^7(0, 1, 1)$ & $2^2(19, 16, 8, 11, 10)$ & $2^4(6, 4, 1, 0, 0, 1, 4)$ & $2(4, 13, 13, 5, 12, 14, 22, 6, 19, 10, 10)$ \\
$10$ & $2^8(1, 1)$ & $2^7(2, 1, 1)$ & $2^8(0, 1, 0, 0, 1)$ & $2^3(20, 9, 7, 6, 6, 7, 9)$ & $2(27, 25, 35, 14, 16, 38, 19, 20, 13, 24, 25)$ \\
$11$ & $2^9(1, 1)$ & $2^6(6, 5, 5)$ & $2^6(6, 4, 1, 1, 4)$ & $2^2(56, 40, 21, 39, 39, 21, 40)$ & $2^9(0, 1, 0, 0, 0, 0, 0, 0, 0, 0, 1)$ \\
$12$ & $2^{10}(1, 1)$ & $2^9(2, 1, 1)$ & $2^4(38, 30, 15, 15, 30)$ & $2(204, 139, 134, 137, 137, 134, 139)$ & $2^9(2, 1, 0, 0, 0, 0, 0, 0, 0, 0, 1)$ \\
$13$ & $2^{11}(1, 1)$ & $2^6(30, 17, 17)$ & $2^3(134, 102, 87, 87, 102)$ & $2(503, 276, 294, 241, 209, 316, 209)$ & $2^8(6, 4, 1, 0, 0, 0, 0, 0, 0, 1, 4)$ \\
$14$ & $2^{12}(1, 1)$ & $2^3(406, 309, 309)$ & $2(999, 855, 716, 666, 860)$ & $2^{12}(0, 1, 0, 0, 0, 0, 1)$ & $2^7(20, 9, 6, 6, 0, 1, 1, 0, 6, 6, 9)$ \\
$15$ & $2^8(35, 29)$ & $2^{10}(6, 5, 5)$ & $2^{12}(0, 1, 1, 1, 1)$ & $2^{10}(6, 4, 1, 0, 0, 1, 4)$ & $2^6(64, 23, 9, 30, 13, 21, 21, 13, 30, 9, 23)$ \\
\bottomrule
\end{tabular}
\end{center}
\caption{$c_p(n)$ for small values of $p$ and $n$}
\label{tab:A}
\end{table}

\section{Type B}\label{sec:B}

In this section we study the ribbon numbers in type $B$.
We first recall some basic definitions and properties on Coxeter groups of type $B$; see Bj\"orner and Brenti~\cite{BjornerBrenti} for more details.

A \emph{signed permutation} of $[n]$ is a bijection $w$ from $\{\pm1, \ldots, \pm n\}$ to itself such that $w(-i)=-w(i)$ for all $i$; this can be expressed as a word $w = w(1) w(2) \cdots w(n)$, where a negative number $-j$ is often written as $\bar j$.
The signed permutations of $[n]$ form the \emph{hyperoctahedral group} $\SS_n^B$, whose order is $2^n\cdot n!$.
This group can be generated by $s_0, s_1, \ldots, s_{n-1}$, where $s_0:=\bar 1 2\cdots n$ and $s_1, \ldots, s_{n-1}$ are the adjacent transpositions, with relations
\[ \begin{cases}
s_i^2 = 1, \ 0\le i\le n-1, \\
s_0s_1s_0s_1 = s_1s_0s_1s_0, \\
s_is_{i+1}s_i = s_{i+1}s_is_{i+1},\ 1\le i \le n-2, \\
s_is_j = s_j s_i,\ |i-j|>1.
\end{cases} \]
Thus $\SS_n^B$ is the Coxeter group of type $B_n$ for $n\ge2$, whose Coxeter diagram is below.
\[
\xymatrix @C=12pt{
s_0 \ar@{-}[r]^{4} & s_1 \ar@{-}[r] & s_2 \ar@{-}[r] & \cdots \ar@{-}[r] & s_{n-2} \ar@{-}[r] & s_{n-1}
}
\]
Given a signed permutation $w\in\SS_n^B$, its length $\ell(w)$ can be described combinatorially, and with $w(0):=0$ and each $s_i$ identified with $i$, we have the descent set 
\[
D(w) = \{i\in\{0,1,\ldots,n-1\}: w(i)>w(i+1) \}.
\]

A sequence $\alpha=(\alpha_1,\ldots,\alpha_\ell)$ of integers with $\alpha_1\ge0$, $\alpha_2,\ldots,\alpha_\ell >0$, and $|\alpha|:=\alpha_1+\cdots+\alpha_\ell=n$ is called a \emph{pseudo-composition} of $n$ and written as $\alpha\models_0 n$.
There is a bijection sending $\alpha \models_0 n$ to its \emph{descent set}
\[ 
D(\alpha):= \{\alpha_1, \alpha_1+\alpha_2, \ldots, \alpha_1+\cdots+\alpha_{\ell-1}\} \subseteq \{0,1,\ldots,n-1\}, 
\]
which can also be encoded in a binary string $a_0a_1\cdots a_n$ with $a_i = 1$ if and only if $i\in D(\alpha)\cup\{n\}$.
Thus there are exactly $2^n$ pseudo-compositions of $n$, and they form a poset under the reverse refinement $\preccurlyeq$, which is isomorphic to the Boolean algebra of subsets of $\{0,1,\ldots, n-1\}$ via $\alpha\mapsto D(\alpha)$.

Similarly to type $A$, the group algebra of $\SS_n^B$ has a deformation called the \emph{type $B$ $0$-Hecke algebra} $H_n^B(0)$, which is generated by $\pi_0, \pi_1, \ldots, \pi_{n-1}$ with relations
\[ \begin{cases}
\pi_i^2 = \pi_i, \ 0\le i\le n-1, \\
\pi_0\pi_1\pi_0\pi_1 = \pi_1\pi_0\pi_1\pi_0, \\
\pi_i\pi_{i+1}\pi_i = \pi_{i+1}\pi_i\pi_{i+1},\ 1\le i \le n-2, \\
\pi_i\pi_j = \pi_j \pi_i,\ |i-j|>1.
\end{cases} \]
The irreducible modules of $H_n^B(0)$ are indexed by pseudo-compositions $\alpha \models_0 n$, and so are the projective indecomposable modules.
The former are all one dimensional, while the latter have dimensions given by the \emph{type $B$ ribbon numbers}
\[ r_\alpha^B := \left\{ w \in \SS_n^B: D(w) = D(\alpha) \right\}. \]

We need a formula for $r^B_\alpha$ to reduce it modulo a prime $p$.

\begin{proposition}\label{prop:RibbonB}
For each pseudo-composition $\alpha$ of $n$, we have
\[ r_\alpha^B = \sum_{\beta\preccurlyeq \alpha} (-1)^{\ell(\alpha)-\ell(\beta)} 2^{n-\beta_1} \binom{n}{\beta}. \]
\end{proposition}

\begin{proof}
Given a pseudo-composition $\beta=(\beta_1,\ldots,\beta_\ell)$ of $n$, the number of signed permutations in $\SS_n^B$ with descent set contained in $D(\beta)$ is
\[ \frac{ 2^n n! } { 2^{\beta_1} \beta_1! \cdots \beta_\ell!} = 2^{n-\beta_1} \binom{n}{\beta}. \]
It is straightforward to prove this either by a direct combinatorial argument or using the quotient of $\SS_n^B$ by its parabolic subgroup indexed by the pseudo-composition $\beta^c$ of $n$ with $D(\beta^c) = \{0,1,\ldots,n\}\setminus D(\beta)$.
The desired formula then follows from inclusion-exclusion.
\end{proof}

Using the same strategy as in type $A$, we can determine the \emph{type $B$ composition dimension $p$-vector} $c^B_p(n):=\left(c^B_{p,i}(n): i\in\mathbb{Z}_p\right)$, where
\[ c^B_{p,i}(n) := \left|\left\{ \alpha\models_0 n: r_\alpha^B \equiv i \pmod p \right\}\right|. \]

We first solve the case $p=2$.

\begin{corollary}
We have $c_{2,0}(n)=0$ and $c_{2,1}(n)=2^n$, i.e., $r_\alpha^B$ is odd for every pseudo-composition $\alpha$ of $n$.
\end{corollary}

\begin{proof}
Let $\alpha$ be a pseudo-composition of $n$. 
By Proposition~\ref{prop:RibbonB}, $r_\alpha^B$ is a sum over pseudo-compositions $\beta\preccurlyeq \alpha$, where the summand indexed by $\beta$ is odd if and only if $\beta_1=n$, i.e., $\beta=(n)$.
Thus $r_\alpha^B$ is odd.
\end{proof}

From now on, we may assume that $p$ is an odd prime.

\begin{theorem}\label{thm:B}
Let $p$ be an odd prime and $n\ge2$ an integer with $[n]_p=(n_0,n_1,\ldots,n_k)$.
Define
\[ P := \left\{ 
b_0+b_1 p+\cdots+b_k p^k: 0\le b_j\le n_j,\ j=0,1,\ldots,k 
\right \}
\setminus \{n\}, \]
which is a subset of $\{0,1,\ldots,n-1\}$ with $|P| = (n_0+1)\cdots(n_k+1)-1$.
For each $T\subseteq P$, define 
\[ 
r^B(T) := \sum_{ \substack{ \beta\models_0 n,\ D(\beta) \subseteq T \\
\beta_{1j}+\cdots+\beta_{\ell j} = n_j,\ \forall j}}
(-1)^{|T|-|D(\beta)|} 
\prod_{j=0}^k 2^{n_j-\beta_{1j}} \binom{n_j}{\beta_{1j}, \ldots, \beta_{\ell j}}.
\]
Here $\beta=(\beta_1,\ldots,\beta_\ell)$ is a pseudo-composition of $n$ with $[\beta_i]_p = (\beta_{i0}, \beta_{i1}, \ldots, \beta_{id})$ for $i=1,\ldots,\ell$.
Then
\[ c^B_{p,i}(n) = 
\begin{cases}
2^{n+1-(n_0+1)\cdots(n_k+1)} \left| \{ T\subseteq P: r^B(T)\equiv i\pmod p \} \right| & \text{if $i=0$ or $n_0=\cdots=n_{k-1}=p-1$} \\
2^{n-(n_0+1)\cdots(n_k+1)} \left| \{ T\subseteq P: r^B(T)\equiv \pm i\pmod p \} \right| & \text{otherwise}.
\end{cases} \]
\end{theorem}

\begin{proof}
Let $\alpha$ be a pseudo-composition of $n$, whose corresponding binary string is $a = a_0a_1\cdots a_n$ with $a_n=1$.
We use Proposition~\ref{prop:RibbonB} to reduce $r^B_\alpha$ modulo $p$.
We have $\beta \preccurlyeq \alpha$ if and only if $a_r=1$ for all $r$ in
\[ 
\left\{ \sum_{i=1}^s \sum_{j=0}^k \beta_{ij}p^j: s=1,\ldots\ell-1 \right\}. 
\]
Moreover, if $\binom{n}{\beta}\not\equiv 0\pmod p$ then $\beta_{1j}+\cdots+\beta_{\ell j} = n_j$ for all $j=0,1,\ldots,k$ by Theorem~\ref{thm:Dickson}, and this implies $D(\beta)\subseteq P$.
Thus to find which pseudo-compositions $\beta$ with $\binom{n}{\beta}\not\equiv 0\pmod p$ are refined by $\alpha$, it suffices to look at the substring $\hat a := (a_r: r\in P)$ of $a$.
It is easy to see that $P$ is a subset of $\{0,1,\ldots,n-1\}$ with $|P|=(n_0+1)\cdots(n_k+1)-1$.

Fix any binary string $b$ indexed by $P$ with $\supp(b) := \{r\in P: b_r=1\}=T$, and suppose $\hat a = b$.
Then we have exactly $2^{n-|P|}$ possibilities for $\alpha$, half of which have even lengths by toggling $a_j$ for some $j\in\{0,1,\ldots,n-1\}\setminus P$ unless $P=\{0,1,\ldots,n-1\}$, i.e., $n_0=\cdots=n_{k-1}=p-1$.
We also have
\[
n-\beta_1 = \sum_{j=0}^k (n_j-\beta_{1j})p^j
\ \Longrightarrow\
2^{n-\beta_1} \equiv \prod_{j=0}^k 2^{n_j-\beta_{1j}} \pmod p \]
by Fermat's little theorem.
Combining this with Proposition~\ref{prop:RibbonB} and Theorem~\ref{thm:Dickson} we obtain
\[
r^B_\alpha \equiv \sum_{ \substack{ \beta\models_0 n \\ 
D(\beta) \subseteq T } } 
(-1)^{\ell(\alpha)-\ell(\beta)} \prod_{j=0}^k 2^{n_j-\beta_{1j}} \binom{n_j}{\beta_{1j}, \ldots, \beta_{\ell j}} 
\equiv r^B(T) \pmod p.
\]
The result follows.
\end{proof}

We derive some consequences of Theorem~\ref{thm:B} below.

\begin{corollary}\label{cor:B}
Let $p$ be an odd prime and $n\ge2$ an integer with $[n]_p=(n_0,\ldots,n_k)$.
For all $i\in\mathbb{Z}_p$ we have $c^B_{p,i}(n)=c^B_{p,-i}(n)$ unless $n_0=\cdots=n_{k-1}=p-1$ and 
\[ \text{$c^B_{p,i}(n)$ is divisible by }
\begin{cases}
2^{n+2-(n_0+1)\cdots(n_k+1)}, & \text{if $i=0$ or $n_0=\cdots=n_{k-1}=p-1$}, \\
2^{n+1-(n_0+1)\cdots(n_k+1)}, & \text{otherwise}.
\end{cases} \]
\end{corollary}

\begin{proof}
Theorem~\ref{thm:B} immediately implies that $c^B_{p,i}(n)=c^B_{p,-i}(n)$ for all $i\in\mathbb{Z}_p$ unless $n_0=\cdots=n_{k-1}=p-1$.
The symmetry $r_I^S=r_{S\setminus I}^S$ for the ribbon numbers of a finite Coxeter system $(W,S)$ mentioned in Section~\ref{sec:prelim} implies that $r^B(T) \equiv r^B(P\setminus T) \pmod p$.
Thus $c^B_{p,i}(n)$ is divisible by the desired power of $2$.
\end{proof}

We can make Theorem~\ref{thm:B} more explicit in some special situations.
We begin with the case when $n$ is a multiple of a power of an odd prime $p$.

\begin{corollary}\label{cor:PowerB}
If $n=mp^d$, where $p$ is an odd prime, $m\in\{1,\ldots,p-1\}$ and $d\ge0$ is an integer, then
\[ c^B_{p,i}(n) =
\begin{cases}
2^{n-m} |\{\gamma\models_0 m: r^B_\gamma \equiv i \pmod p\}| & \text{if $i=0$ or $d=0$}, \\
2^{n-m-1} |\{\gamma\models_0 m: r^B_\gamma \equiv \pm i \pmod p\}| & \text{otherwise}.
\end{cases}\]
\end{corollary}

\begin{proof}
By Theorem~\ref{thm:B}, we have $[n]_p=(0,\ldots,0,m)$, $P=\{jp^d: j=0,1,\ldots,m-1\}$, and each $T\subseteq P$ corresponds to a pseudo-composition $\gamma\models_0 m$ with descent set $D(\gamma)=\{t/p^d: t\in T\}$ such that
\begin{align*}
r^B(T) & = \sum_{ \substack{ \beta\models_0 n \\ D(\beta) \subseteq T } } 
(-1)^{|T|-|D(\beta)|} \prod_{j=0}^d  2^{n_j-\beta_{1j}} \binom{n_j}{\beta_{1j}, \ldots, \beta_{\ell j}} \\
&= \sum_{ \delta\preccurlyeq \gamma } 
(-1)^{\ell(\gamma)-|\ell (\delta)|} 2^{k-\delta_1} \binom{k}{\delta} 
= r^B_\gamma,
\end{align*}
where $\delta$ is obtained from $\beta$ by dividing each part by $p^d$.
The result follows.
\end{proof}

\begin{example}
Corollary~\ref{cor:PowerB} becomes trivial when $d=0$. 
Assume $d>0$ below.
For $m=1,2,3$ we can compute $r^B_\gamma$ for all $\gamma\models m$:
\begin{itemize}
\item
$r^B_{(1)} = r^B_{(0,1)} = 1$, 
$r^B_{(2)} = r^B_{(0,1,1)} = 1$,
$r^B_{(1,1)} = r^B_{0,2} = 3$,
\item
$r^B_{(3)} = r^B_{(0,1,1,1)} = 1$, 
$r^B_{(2,1)} = r^B_{(0,1,2)} = 5$,
$r^B_{(0,3)} = r^B_{(1,1,1)} = 7$,
$r^B_{(1,2)} = r^B_{(0,2,1)} = 11$.
\end{itemize} 
Thus we have the following by Corollary~\ref{cor:PowerB}, where $p$ is an odd prime.
\begin{itemize}
\item
Assume $n=p^d$. Then $c_p(n)=(0,2^{n-1},0,\ldots,0,2^{n-1})$.
\item
Assume $n=2p^d$. Then $c_p(n)=2^{n-2}(2,1,1)$ when $p=3$ and $c_{p,i}(n)=2^{n-2}$ if $i\equiv\pm1,\pm3\pmod p$ or $c_{p,i}(n)=0$ otherwise when $p>3$.
\item
Assume $n=3p^d$. Then $c_p(n)=2^{n-3}(2,2,1,1,2)$ when $p=5$, $c_p(n)=2^{n-3}(2,1,1,1,1,1,1)$ when $p=7$, $c_p(n)=2^{n-3}(2,1,0,0,1,1,1,1,0,0,1)$ when $p=11$, and $c_{p,i}(n)=2^{n-3}$ if $i\equiv\pm1,\pm5,\pm7,\pm11 \pmod p$ or $c_{p,i}(n) = 0$ otherwise when $p>11$.
\end{itemize}
\end{example}

Next, we consider the case when $n$ is a sum of distinct powers of an odd prime $p$.

\begin{corollary}\label{cor:PowersB}
Suppose $n=p^{d_1}+\cdots+p^{d_k}$, where $p$ is an odd prime, $0\le d_1<\cdots<d_k$, and $k>1$.
Define
\[ 
P:= \left\{ \overline U: U\subsetneqq \{p^{d_1}, \ldots, p^{d_k} \} \right\} 
\quad \text{ordered by } \preccurlyeq
\] 
where $\overline U$ denotes the sum of all elements of $U$ and $\overline U \preccurlyeq \overline V$ in the poset $P$ if and only if $U\subseteq V$.
Then
\[ c^B_{p,i}(n) = 
\begin{cases}
2^{n-2^d+1} \left| \left\{ T\subseteq P: \chi^B(T) \equiv i \pmod p \right\} \right|, & \text{if $i=0$}, \\
2^{n-2^d}|\{ T\subseteq P: \chi^B(T) \equiv \pm i \pmod p \}|, & \text{otherwise}.
\end{cases} \]
Here $\chi^B(T)$ is the following sum over chains (including the empty one) in $T$ (as a subposet of $P$):
\[ 
\chi^B(T) := \sum_{ \substack{ U_1\subsetneqq \cdots \subsetneqq U_h \subsetneqq U_{h+1}= \left\{p^{d_1}, \cdots, p^{d_k} \right\} \\ \overline{U_1}, \ldots, \overline{U_h}\in T }} (-1)^h 2^{d-|U_1|} 
\]
\end{corollary}

\begin{proof}
By Theorem~\ref{thm:B}, we have $P= \left\{ \overline U: U\subsetneqq \{p^{d_1}, \ldots, p^{d_k} \} \right\} $ with $|P|=2^k-1$, and for every $T\subseteq P$,
\[ 
r^B(T) = (-1)^{|T|}\chi^B(T)
\]
since in the definition of $r^B(T)$, a pseudo-composition $\beta\models_0 n$ satisfies $D(\beta)\subseteq T$ if and only if $D(\beta)$ gives a chain of cardinality $|D(\beta)|$ in $T$, the power $2^{n_j-\beta_{1j}}$ is either $2$ when $n_j=1$ and $\beta_{1j}=0$ or $1$ when $n_j=\beta_{1j}\in\{0,1\}$, and the multinomial coefficients involved are all equal to one.
The result then follows immediately (the case $n_0=\cdots=n_{k-1}=p-1$ does not occur since $k>1$).
\end{proof}

We have the following example when $n$ is a sum of two distinct powers of a prime $p$.

\begin{example}
Suppose $n=u+v$, where $u$ and $v$ are distinct powers of a prime $p$.
By Corollary~\ref{cor:PowersB}, we have $P=\{0,u,v\}$ and for every $T\subseteq P$, we compute $\chi^B(T)$ below.
\begin{itemize}
\item
The only chain in $T=\emptyset$ is the empty chain, so $\chi^B(\emptyset) = 1$.
\item
The only nonempty chain in $T=\{0\}$ is $0=\overline\emptyset$, so $\chi^B(\{0\}) = 1-2^2=-3$.
\item
The only nonempty chain in $T=\{u\}$ is $u$, so $\chi^B(\{u\})=1-2=-1$.
\item
The only nonempty chains in $T=\{0,u\}$ are $0$, $u$, and $0\preccurlyeq u$, so $\chi^B(\{0,u\})=1-2^2-2+2^2=-1$.
\item
The only nonempty chains in $T=\{u,v\}$ are $u$ and $v$, so $\chi^B(\{u,v\})=1-2-2=-3$.
\item
The only nonempty chains in $T=\{0,u,v\}$ are $0$, $u$, $v$, $0\preccurlyeq u$, and $0\preccurlyeq v$, so $\chi^B(\{0,u,v\}) = 1-2^2-2-2+2^2+2^2=1$.
\end{itemize}
Note that swapping $u$ and $v$ does not change $\chi^B(T)$.
It follows that $c^B_p(n)=2^{n-3}(2,3,3)$ when $p=3$, and $c^B_{p,\pm1}(n)=3\cdot2^{n-3}$, $c^B_{p,\pm3}(n)=2^{n-3}$, and $c^B_{p,i}(n)=0$ if $i\not\equiv\pm1,\pm3\pmod p$ when $p>3$.
For example, we have $c^B_3(4)=(4,6,6)$, $c^B_3(10)=(256, 384, 384)$, $c^B_5(6)=(0,24,8,8,24)$, and $c^B_7(8)=(0,96,0,32,32,0,96)$.
\end{example}

We compute $c^B_p(n)$ based on its definition for small values of $n$ and $p$ in Sage and give our data in Table~\ref{tab:B}, which agree with the results in this section; 
note that the power of $2$ given by Corollary~\ref{cor:B} may or may not be the highest in $c^B_{p,i}(n)$. 

\begin{table}[h]
\tiny
\renewcommand{\arraystretch}{1.2}
\begin{center}
\begin{tabular}{c|ccccc} 
\toprule
$n$ & $p=3$ & $p=5$ & $p=7$ & $p=11$ \\ 
\midrule
$2$ & $2(1, 1, 0)$ & $2(0, 1, 0, 1, 0)$ & $2(0, 1, 0, 1, 0, 0, 0)$ & $2(0, 1, 0, 1, 0, 0, 0, 0, 0, 0, 0)$ \\
$3$ & $2^2(0, 1, 1)$ & $2(1, 2, 1, 0, 0)$ & $2(1, 1, 0, 0, 1, 1, 0)$ & $2(1, 1, 0, 0, 0, 1, 0, 1, 0, 0, 0)$ \\
$4$ & $2(2, 3, 3)$ & $2(1, 3, 3, 1, 0)$ & $2(1, 3, 1, 2, 0, 0, 1)$ & $2(0, 2, 1, 0, 1, 0, 1, 1, 1, 1, 0)$ \\
$5$ & $2(4, 9, 3)$ & $2^4(0, 1, 0, 0, 1)$ & $2(4, 1, 4, 2, 4, 0, 1)$ & $2(3, 1, 2, 0, 0, 1, 1, 2, 1, 5, 0)$ \\
$6$ & $2^4(2, 1, 1)$ & $2^3(0, 3, 1, 1, 3)$ & $2(4, 7, 4, 6, 5, 2, 4)$ & $2(4, 7, 1, 2, 2, 3, 2, 1, 9, 1, 0)$ \\
$7$ & $2^4(2, 3, 3)$ & $2^2(6, 7, 6, 6, 7)$ & $2^6(0, 1, 0, 0, 0, 0, 1)$ & $2(7, 3, 6, 5, 7, 9, 8, 7, 5, 6, 1)$ \\
$8$ & $2(50, 39, 39)$ & $2^3(6, 8, 5, 5, 8)$ & $2^5(0, 3, 0, 1, 1, 0, 3)$ & $2(13, 13, 11, 13, 8, 10, 11, 12, 13, 10, 14)$ \\
$9$ & $2^8(0, 1, 1)$ & $2(52, 59, 49, 56, 40)$ & $2^4(4, 6, 3, 5, 5, 3, 6)$ & $2(29, 15, 29, 25, 19, 19, 26, 21, 37, 21, 15)$ \\
$10$ & $2^7(2, 3, 3)$ & $2^8(0, 1, 1, 1, 1)$ & $2^3(26, 19, 19, 13, 13, 19, 19)$ & $2(59, 46, 51, 47, 43, 47, 41, 50, 37, 51, 40)$ \\
$11$ & $2^8(2, 3, 3)$ & $2^6(6, 7, 6, 6, 7)$ & $2^2(58, 91, 62, 74, 74, 62, 91)$ & $2^{10}(0, 1, 0, 0, 0, 0, 0, 0, 0, 0, 1)$ \\
$12$ & $2^9(2, 3, 3)$ & $2^6(6, 12, 17, 17, 12)$ & $2(248, 256, 307, 337, 337, 307, 256)$ & $2^9(0, 3, 0, 1, 0, 0, 0, 0, 1, 0, 3)$ \\
$13$ & $2^7(18, 23, 23)$ & $2^2(458, 440, 355, 355, 440)$ & $2(570, 696, 516, 565, 571, 525, 653)$ & $2^8(2, 6, 0, 4, 2, 3, 3, 2, 4, 0, 6)$ \\
$14$ & $2^5(166, 173, 173)$ & $2(1523, 1775, 1647, 1567, 1680)$ & $2^{12}(0, 1, 0, 1, 1, 0, 1)$ & $2^7(14, 14, 13, 9, 10, 11, 11, 10, 9, 13, 14)$ \\
$15$ & $2^{12}(2, 3, 3)$ & $2^{12}(2, 2, 1, 1, 2)$ & $2^{10}(4, 6, 3, 5, 5, 3, 6)$ & $2^6(44, 48, 49, 40, 52, 45, 45, 52, 40, 49, 48)$ \\ 
\bottomrule
\end{tabular}
\end{center}
\caption{$c^B_p(n)$ for small values of $p$ and $n$}
\label{tab:B}
\end{table}

\section{Type D}\label{sec:D}

Now we study the ribbon numbers in type $D$. 
The reader is referred to Bj\"orner and Brenti~\cite{BjornerBrenti} for details on Coxeter groups of type $D$.

A signed permutation $w\in\SS_n^B$ is \emph{even} (resp., \emph{odd}) if the number of negative in $w(1), w(2), \ldots, w(n)$ is even (resp., odd).
The even signed permutations in $\SS_n^B$ form a subgroup $\SS_n^D$, which is generated by $s_0 := \bar 2\bar 1 3\cdots n$ (different from $s_0$ in type $B$) and the adjacent transpositions $s_1,\ldots,s_{n-1}$ with the relations
\[ \begin{cases}
s_i^2 = 1, \ 0\le i\le n-1, \\
s_0 s_2 s_0 = s_2 s_0 s_2, \\
s_is_{i+1}s_i = s_{i+1}s_is_{i+1},\ 1\le i \le n-2, \\
s_is_j = s_j s_i,\ |i-j|>1.
\end{cases} \]
Thus $\SS_n^D$ is the Coxeter group of type $D_n$ for $n\ge4$, whose Coxeter diagram is given below.
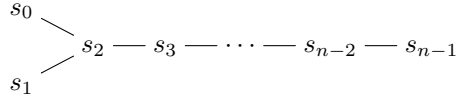
\[ \xymatrix @R=2pt @C=12pt{
s_0 \ar@{-}[rd] \\
& s_2 \ar@{-}[r] & s_3 \ar@{-}[r] & \cdots \ar@{-}[r] & s_{n-2} \ar@{-}[r] & s_{n-1} \\
s_1 \ar@{-}[ru]
} \]

We have $|\SS_n^D|=2^{n-1}n!$ since toggling the sign of $w(n)$ gives a bijection between even and odd signed permutations of $[n]$.
For each $w\in\SS_n^D$, the combinatorial interpretation of the length of $w$ in $\SS_n^D$ is slightly different from its length in $\SS_n^B$, but its descent set can be described in a similar way as in type $B$ with the convention that $w(0):=-w(2)$:
\[ D(w)= \{i\in\{0,1,\ldots,n-1: w(i)>w(i+1) \} \]

The \emph{type $D$ $0$-Hecke algebra} $H_n^D(0)$ is generated by $\pi_0, \pi_1, \ldots, \pi_2$; the relations satisfies by these generators are the same as the above relations for $s_0, s_1, \ldots, s_{n-1}$ except that $\pi_i^2 = \pi_i$ for $i=0,1,\ldots,n-1$.
Both irreducible modules and projective indecomposable modules of $H_n^D(0)$ are indexed by pseudo-compositions $\alpha\models_0 n$.
The former are all one dimensional, whereas the latter have dimensions given by the \emph{type $D$ ribbon numbers}
\[ 
r_\alpha^D := |\{w\in\SS_n^D: D(w) = D(\alpha) \}|. 
\]

To study $r_\alpha^D$, we need the following formula.

\begin{proposition}\label{prop:RibbonD}
For each pseudo-composition $\alpha$ of $n$, we have
\[ r_\alpha^D = \sum_{\beta\preccurlyeq \alpha} (-1)^{\ell(\alpha)-\ell(\beta)} \nu(\beta), \]
where
\[ \nu(\beta) :=
\begin{cases}
2^{n-1}\binom{n}{\beta}, & \text{if } \beta_1=0, \\
2^{n-1}\binom{n}{1+\beta_2,\beta_3,\ldots,\beta_\ell}, & \text{if } \beta_1=1, \\
2^{n-\beta_1}\binom{n}{\beta}, & \text{if } 1<\beta_1\le n.
\end{cases} \]
\end{proposition}

\begin{proof}
It suffices to show that the number of signed permutations in $\SS_n^D$ with descent set contained in $D(\beta)$ is $\nu(\beta)$ for every pseudo-composition $\beta$ of $n$.
This can be proved by considering the quotient of $\SS_n^D$ by its parabolic subgroup generated by $\{s_i: i\in\{0,1,\ldots,n-1\}\setminus D(\beta)\}$.
Alternatively, the following case-by-case combinatorial argument on $|\{w\in \SS_n^D: D(w)\subseteq D(\beta)\}|$ works.

\vskip3pt\noindent \textsf{Case 1}: $\beta_1=0$.
Then $w\in\SS_n^D$ has $D(w)\subseteq D(\beta)$ if and only if 
\[ 
w(1)<\cdots<w(\beta_2),\ 
w(\beta_2+1)<\cdots<w(\beta_2+\beta_3),\ 
\cdots,\ 
w(\beta_2+\cdots+\beta_{\ell-1}+1)<\cdots<w(n). 
\]
There are $2^n \binom{n}{\beta}$ signed permutations $w\in\SS_n^B$ satisfying the above, half of which belong to $\SS_n^D$ by toggling $w(n)$.
Thus the number of signed permutations $w\in \SS_n^D$ belonging to this case is $2^{n-1} \binom{n}{\beta}$.

\vskip3pt\noindent \textsf{Case 2}: $\beta_1=1$.
Then $w\in\SS_n^D$ has $D(w)\subseteq D(\beta)$ if and only if $-w(2)<w(1)$ and
\[ 
w(2)<\cdots<w(1+\beta_2),\
w(2+\beta_2)<\cdots<w(1+\beta_2+\beta_3),\
\ldots,\
w(2+\beta_2+\cdots+\beta_{\ell-1})<\cdots<w(n). 
\]
We may replace $-w(2)<w(1)$ with $|w(1)|<w(2)$ or $|w(2)|<w(1)$.

If the former holds, then $0<|w(1)|<w(2)<\cdots<w(1+\beta_2)$ and the sign of $w(1)$ is determined by the signs of $w(i)$ for all $i>1+\beta_2$, so the number of possibilities for $w$ is 
\[ 
\frac{2^{n-1-\beta_2}n!}{(1+\beta_2)!\beta_3!\cdots\beta_\ell!}. 
\]

If the latter holds, then for each $i\in\{3,\ldots,1+\beta_2\}$ with $|w(i)|>w(1)$, we must have $w(i)>0$ (otherwise $w(i)<-w(1)<w(2)$), so the number of possibilities for $w$ is 
\[
\sum_{j=0}^{\beta_2-1} \frac{2^{n-2-j}n!}{(1+\beta_2)!\beta_3!\cdots\beta_\ell!}  
= \frac{2^{n-1-\beta_2}(2^{\beta_2}-1)n!}{(1+\beta_2)!\beta_3!\cdots\beta_\ell!}.
\]
Here $j:= \{i: |w(i)|>w(1),\ 3\le i\le 1+\beta_2\}$.

Adding the above two results, we have the number of signed permutations $w\in \SS_n^D$ belonging to this case is $2^{n-1}\binom{n}{1+\beta_2,\beta_3,\ldots,\beta_\ell}$.

\vskip3pt\noindent \textsf{Case 3}: $\beta_1>1$.
Then $w\in\SS_n^D$ has $D(w)\subseteq D(\beta)$ if and only if $-w(2)<w(1)<w(2)<\cdots<w(\beta_1)$ and
\[ 
|w(1)|<w(2)<\cdots<w(\beta_1),\
w(\beta_1+1)<\cdots<w(\beta_1+\beta_2),\
\ldots,\
w(\beta_1+\cdots+\beta_{\ell-1}+1)<\cdots<w(n). 
\]
Note that the sign of $w(1)$ is determined by the signs of $w(i)$ for all $i>\beta_1$.
Thus the number of signed permutations $w\in \SS_n^D$ belonging to this case $2^{n-\beta_1} \binom{n}{\beta}$.
\end{proof}

Using the same strategy for type A and type B, we determine the \emph{type $D$ composition dimension $p$-vector} $c^D_p(n):=\left(c^D_{p,i}(n): i\in\mathbb{Z}_p\right)$, where 
\[ 
c^D_{p,i}(n) := |\{\alpha\models_0 n: r^D_\alpha \equiv i\pmod p\}|.
\]

We first settle the case $p=2$.

\begin{corollary}
If $n\ge4$ then $c^D_2(n)=(0,2^n)$, i.e., $r_\alpha^D$ is odd for all $\alpha\models_0 n$. 
\end{corollary}

\begin{proof}
Let $\alpha$ be a pseudo-composition of $n$. 
Then $r_\alpha^D$ is odd by Proposition~\ref{prop:RibbonD}, since for each $\beta\preccurlyeq \alpha$, we have $\nu(\beta)$ is odd if and only if $\beta_1=n$, i.e., $\beta=(n)$.
\end{proof}

From now on we may assume that $p$ is an odd prime.

\begin{theorem}\label{thm:D}
Let $p$ be an odd prime and $n\ge4$ an integer with $[n]_p=(n_0,n_1,\ldots,n_k)$.
For each pseudo-composition $\beta=(\beta_1,\ldots,\beta_\ell)$ of $n$ with $[\beta_i]_p = (\beta_{i0}, \beta_{i1}, \ldots, \beta_{id})$, 
\begin{itemize}
\item
if $\beta_1=0$ then define $\beta':=\beta$ and $\nu_p(\beta'):= \frac12\prod_{j=0}^d 2^{n_j} \binom{n_j}{\beta_{1j}, \ldots, \beta_{\ell j}}$;
\item
if $\beta_1=1$ then define $\beta':=(0,1+\beta_2, \beta_3,\ldots, \beta_\ell) \models_0 n$ 
and $\nu_p(\beta')$ as in the last case; 
\item
if $\beta_1>1$ then define $\beta':=\beta$ and $\nu_p(\beta') := \prod_{j=0}^d 2^{n_j-\beta_{1j}} \binom{n_j}{\beta_{1j}, \ldots, \beta_{\ell j}}$.
\end{itemize}
For each $T\subseteq P :=  \left\{ b_0+b_1 p+\cdots+b_k p^d: 0\le b_j\le n_j,\ j=0,1,\ldots,k \right \} \cup \{1\} \setminus \{n\} \subseteq \{0,1,\ldots,n-1\}$, let
\[ 
r^D(T) := \sum_{ \substack{ \beta\models_0 n,\ D(\beta) \subseteq T \\ 
\beta'_{1j}+\cdots+\beta'_{\ell j} = n_j,\ \forall j}}
(-1)^{|T|-|D(\beta)|} \nu_p(\beta').
\]
Then $|P|=(n_0+1)\cdots(n_k+1)$ if $b_0=0$ or $|P|=(n_0+1)\cdots(n_k+1)-1$ of $b_0>0$, and 
\[ c^D_{p,i}(n) = 
\begin{cases}
2^{n-|P|} \left| \{ T\subseteq P: r^D(T)\equiv i\pmod p \} \right| & \text{if $i=0$ or $n_0=\cdots=n_{k-1}=p-1$} \\
2^{n-|P|-1} \left| \{ T\subseteq P: r^D(T)\equiv \pm i\pmod p \} \right| & \text{otherwise}.
\end{cases} \]
\end{theorem}

\begin{proof}
We determine $c^D_p(n)$ by using Proposition~\ref{prop:RibbonD} to reduce $r^D_\alpha$ modulo $p$ for an arbitrary pseudo-composition $\alpha$ of $n$.
Let $a=a_0a_1\cdots a_n$ be the binary string corresponding to $\alpha$, that is, $a_i=1$ if and only if $i\in D(\alpha)\cup \{n\}$.
We have $\beta \preccurlyeq \alpha$ if and only if $a_r=1$ for all $r$ in
\[ 
\left\{ \sum_{i=1}^s \sum_{j=0}^d \beta_{ij}p^j: s=1,\ldots\ell-1 \right\}. 
\]
If $\nu(\beta)\not\equiv 0\pmod p$, then $\beta'_{1j}+\cdots+\beta'_{\ell j} = n_j$ for all $j=0,1,\ldots,k$ by Theorem~\ref{thm:Dickson}, and this implies $D(\beta)\subseteq P$.
Thus to find which pseudo-compositions $\beta$ with $\nu(\beta)\not\equiv 0\pmod p$ are refined by $\alpha$, it suffices to look at the substring $\hat a := (a_r: r\in P)$ of $a$.
It is easy to check that $P$ is a subset of $\{0,1,\ldots,n-1\}$ with 
\[ |P|=\begin{cases}
(n_0+1)\cdots(n_k+1), & \text{if $1\notin P$, i.e., $b_0=0$}, \\
(n_0+1)\cdots(n_k+1)-1, & \text{if $1\in P$, i.e., $b_0>0$}.
\end{cases} \] 

Let $b$ be a fixed binary string indexed by $P$ with $\supp(b) := \{r\in P: b_r=1\}=T$, and suppose $\hat a = b$.
We have exactly $2^{n-|P|}$ possibilities for $\alpha$, half of which have even lengths by toggling $a_j$ for some $j\in\{0,1,\ldots,n-1\}\setminus P$ unless $P=\{0,1,\ldots,n-1\}$, i.e., $n_0=\cdots=n_{k-1}=p-1$.
By Fermat's little theorem, we have
\[
n-1 = -1 + \sum_{j=0}^d n_j p^j \
\Longrightarrow \ 
2^{n-1} \equiv \frac12 \prod_{j=0}^d 2^{n_j} \pmod p,
\]
\[
n-\beta_1 = \sum_{j=0}^d (n_j-\beta_{1j})p^j
\ \Longrightarrow\
2^{n-\beta_1} \equiv \prod_{j=0}^d 2^{n_j-\beta_{1j}} \pmod p. \]
Combining this with Proposition~\ref{prop:RibbonD} and Theorem~\ref{thm:Dickson} we obtain
\[
r^D_\alpha \equiv \sum_{ \substack{ \beta\models_0 n,\ D(\beta) \subseteq T \\ 
\beta'_{1j}+\cdots+\beta'_{\ell j} = n_j,\ \forall j}}
(-1)^{\ell(\alpha)-\ell(\beta)} \nu_p(\beta')
\equiv r^D(T) \pmod p.
\]
The result follows.
\end{proof}

We derive some consequences of Theorem~\ref{thm:B} below.

\begin{corollary}\label{cor:D}
For all $i\in\mathbb{Z}_p$, we have $c^D_{p,i}(n)=c^D_{p,-i}(n)$ unless $n_0=\cdots=n_{k-1}=p-1$ and 
\[ \text{$c^D_{p,i}(n)$ is divisible by }
\begin{cases}
2^{n+2-(n_0+1)\cdots(n_k+1)}, & \text{ if ($n_0>0$ and $i=0$) or $n_0=\cdots=n_{k-1}=p-1$}, \\
2^{n-(n_0+1)\cdots(n_k+1)}, & \text{ if $n_0=0$ and $i\ne 0$}, \\
2^{n+1-(n_0+1)\cdots(n_k+1)}, & \text{otherwise}.
\end{cases} \]
\end{corollary}

\begin{proof}
Theorem~\ref{thm:D} immediately implies that $c^D_{p,i}(n)=c^D_{p,-i}(n)$ for all $i\in\mathbb{Z}_p$ unless $n_0=\cdots=n_{k-1}=p-1$.
The symmetry $r_I^S=r_{S\setminus I}^S$ for the ribbon numbers of a finite Coxeter system $(W,S)$ mentioned in Section~\ref{sec:prelim} implies $r(T) \equiv r(P\setminus T) \pmod p$.
Thus $c^D_{p,i}(n)$ is divisible by $2^{n-|P|+1}$ if $i=0$ or $n_0=\cdots=n_{k-1}=p-1$ or divisible by $2^{n-|P|}$ otherwise.
We also have $|P|=(n_0+1)\cdots(n_k+1)$ if $n_0=0$ or $|P|=(n_0+1)\cdots(n_k+1)-1$ otherwise.
The result follows.
\end{proof}

We can make Theorem~\ref{thm:D} more explicit in some special situations.
We begin with case when $n$ is a small multiple of a power of $p$.

\begin{corollary}
Let $p$ be an odd prime. The following holds for $c^D_p(n)$, where $d>0$ is an integer.
\begin{itemize}
\item[(i)]
If $n=p^d$ then $c^D_{p,0}(n) = 2^{n-1}$, $c^D_{p,\pm1}(n) = 2^{n-2}$, and $c^D_{n,i}(n) = 0$ for all $i\not\equiv 0,\pm1\pmod p$.
\item[(ii)]
If $n=2p^d$ then $c^D_{p,\pm1}(n)=3\cdot 2^{n-3}$, $c^D_{p,\pm3}(n)=2^{n-3}$, and $c^D_{n,i}(n) = 0$ for all $i\not\equiv 0,\pm1,\pm3\pmod p$ when $p>3$ and $c^D_3(n)=(2^{n-2},3\cdot 2^{n-3},3\cdot 2^{n-3})$.
\item[(iii)]
If $n=3p^d$ and $p=5$ then $c^D_5(n)=(2^{n-3},2^{n-3},5\cdot2^{n-4},5\cdot2^{n-4},2^{n-3})$.

\noindent
If $n=3p^d$ and $p=7$ then $c^D_7(n)=(2^{n-3},2^{n-3},5\cdot2^{n-4},5\cdot2^{n-4},2^{n-3})$.

\noindent
If $n=3p^d$ and $p=11$ then $c^D_{11}(n)=(2^{n-3},2^{n-4},0,2^{n-2},2^{n-4},2^{n-4},2^{n-4},2^{n-4},2^{n-2},0,2^{n-4})$.

\noindent
If $n=3p^d$ and $p>11$ then $c^D_{p,\pm1}(n) = 2^{n-4}$, $c^D_{p,\pm3}(n)=2^{n-2}$, $c^D_{p,\pm5}(n)=2^{n-4}$, $c^D_{p,\pm7}(n)=2^{n-4}$, $c^D_{p,\pm11}(n)=2^{n-4}$, and $c^D_{p,i}(n)=0$ for all $i\not\equiv \pm1,\pm3,\pm5,\pm7,\pm11\pmod p$.
\end{itemize}
\end{corollary}

\begin{proof} 
We apply Theorem~\ref{thm:D} to the following cases.

(i) Suppose $n=p^d$ for some integer $d>0$.
We compute $r^D(T)$ for every $T\subseteq P:=\{0,1\}$.
For each $\beta$ in the definition of $r^D(T)$, we have
\[ \nu_p(\beta')=
\begin{cases}
1 & \text{if } \beta\in\{(n)\}; \\
1-1-1=-1 & \text{if } \beta\in\{(0,n), (1,n-1)\}.
\end{cases} \]
Thus 
\[ r^D(T) =  
\begin{cases}
1-1-1=-1, & \text{if } T=\{0,1\};\\
(-1)(1-1)=0, & \text{if } T=\{0\} \text{ of } \{1\}; \\
1, & \text{if } T= \emptyset. 
\end{cases} \]
It follows that $c^D_{p,0}(n) = 2^{n-1}$, $c^D_{p,\pm1}(n) = 2^{n-2}$, and $c^D_{n,i}(n) = 0$ for all $i\not\equiv 0,\pm1\pmod p$.

(ii) Suppose $n=2p^d$ for some integer $d>0$.
We compute $r^D(T)$ for every $T\subseteq P:=\{0,1,p^d\}$.
For each $\beta$ in the definition of $r^D(T)$, we have
\[ \nu_p(\beta')=
\begin{cases}
1 & \text{if } \beta\in\{(n)\}; \\
2 & \text{if } \beta\in\{(0,n), (1,n-1)\}; \\
4 & \text{if } \beta\in\{(0,p^d,p^d), (1,p^d-1,p^d), (p^d,p^d)\}.
\end{cases} \]
Thus
\[ r^D(T) =  
\begin{cases}
(-1)(1-2-2-4+4+4)=-1, & \text{if } T=\{0,1,p^d\};\\
1-2-2=-3, & \text{if } T=\{0,1\}; \\
1-2-4+4 = -1, & \text{if } T= \{0,p^d\} \text{ or }\{1,p^d\}; \\
(-1)(1-2) = 1, & \text{if } T= \{0\} \text{ or }\{1\}; \\
(-1)(1-4) = 3, & \text{if } T= \{p^d\}; \\
1, & \text{if } T= \emptyset.
\end{cases} \]
It follows that $c^D_{p,0}(n) = 2^{n-2}$, $c^D_{p,\pm1}(n)=3\cdot 2^{n-3}$, and $c^D_{n,i}(n) = 0$ for all $i\not\equiv 0,\pm1\pmod p$ if $p=3$ and $c^D_{p,\pm1}(n)=3\cdot 2^{n-3}$, $c^D_{p,\pm3}(n)=2^{n-3}$, and $c^D_{n,i}(n) = 0$ for all $i\not\equiv 0,\pm1,\pm3\pmod p$ if $p>3$.

(iii) Suppose $n=3p^d$ for some integer $d>0$.
We compute $r^D(T)$ for every $T\subseteq P:=\{0,1,p^d,2p^d\}$.
For each $\beta$ in the definition of $r^D(T)$, we have
\[ \nu_p(\beta')=
\begin{cases}
1 & \text{if } \beta\in\{(n)\}; \\
4 & \text{if } \beta\in\{(0,n), (1,n-1)\}; \\
6 & \text{if } \beta\in\{(2p^d,p^d)\}; \\
12 & \text{if } \beta\in\{(0,p^d,2p^d), (0,2p^d,p^d),(1,p^d-1,2p^d), (1,2p^d-1,p^d), (p^d,2p^d)\}; \\
24 & \text{if } \beta\in\{(0,p^d,p^d,p^d), (1,p^d-1,p^d,p^d), (p^d,p^d,p^d)\}.
\end{cases} \]
Thus
\[ r^D(T) =  
\begin{cases}
1-4\cdot2-6-12+12\cdot4+24-24\cdot2=-1, & \text{if } T=\{0,1,p^d,2p^d\};\\
(-1)(1-4\cdot2-12+12\cdot2)=-5, & \text{if } T=\{0,1,p^d\}; \\
(-1)(1-4\cdot2-6+12\cdot2)=-11, & \text{if } T=\{0,1,2p^d\}; \\
(-1)(1-4-6-12+12\cdot2+24-24) = -3, & \text{if } T= \{0,p^d,2p^d\} \text{ or }\{1,p^d,2p^d\}; \\
1-4-4=-7, & \text{if } T=\{0,1\}; \\
1-6-12+24=7, & \text{if } T=\{p^d,2p^d\}; \\
1-4-12+12=-3, & \text{if } T= \{0,p^d\} \text{ or }\{1,p^d\}; \\
1-4-6+12=3, & \text{if } T= \{0,2p^d\} \text{ or }\{1,2p^d\}; \\
(-1)(1-4)=3, & \text{if } T= \{0\} \text{ or }\{1\}; \\
(-1)(1-12)=11, & \text{if } T= \{p^d\}; \\
(-1)(1-6)=5, & \text{if } T= \{2p^d\}; \\
1, & \text{if } T= \emptyset.
\end{cases} \]
The result on $c^D_p(n)$ follows.
\end{proof}

Next, we study the case when $n$ is a sum of two distinct powers of $p$.

\begin{corollary}
Let $p$ be an odd prime. The following holds for $c^D_p(n)$.
\begin{itemize}
\item[(i)]
If $n=1+p^d$ for some integer $d>0$ then $c^D_{p,\pm1}(n) = 2^{n-1}$ and $c^D_{p,i}(n)=0$ for all $i\not\equiv \pm1\pmod p$.
\item[(ii)]
If $n$ is the sum of two distinct positive powers of $p$ then 
$c^D_{p,\pm1}(n) = 7\cdot 2^{n-4}$, $c^D_{p,\pm3}(n) = 2^{n-4}$, and $c^D_{p,i}(n)=0$ for all $i\not\equiv \pm1, \pm3\pmod p$ when $p>3$ and $c^D_p(n)=(2^{n-3},7\cdot 2^{n-4}, 7\cdot 2^{n-4})$ when $p=3$.
\end{itemize}
\end{corollary}

\begin{proof}
We apply Theorem~\ref{thm:D} to the following cases.

(i) Suppose $n=1+p^d$ for some integer $d>0$.
We compute $r^D(T)$ for every $T\subseteq P:=\{0,1,p^d\}$.
For each $\beta$ appearing in the definition of $r^D(T)$, we have
\[ \nu_p(\beta')=
\begin{cases}
1 & \text{if } \beta\in\{(n)\}; \\
2 & \text{if } \beta\in\{(0,n),(0,1,p^d),(0,p^d,1),(1,p^d-1,1),(1,p^d),(p^d,1)\}.
\end{cases} \]
Thus
\[ r^D(T) =  
\begin{cases}
(-1)(1-2-2-2+2+2+2)=-1, & \text{if } |T|=3;\\
1-2-2+2=-1, & \text{if } |T|=2; \\
(-1)(1-2)=1, & \text{if } |T|=1; \\
1, & \text{if } |T|=0.
\end{cases} \]
Thus $c^D_{p,\pm1}(n) = 2^{n-1}$ and $c^D_{p,i}(n)=0$ for all $i\not\equiv \pm1\pmod p$.

(ii) Suppose $n=i+j$, where $i$ and $j$ are distinct positive powers of $p$.
We compute $r^D(T)$ for every $T\subseteq P:=\{0,1,i,j\}$.
For each $\beta$ appearing in the definition of $r^D(T)$, we have
\[ \nu_p(\beta')=
\begin{cases}
1 & \text{if } \beta\in\{(n)\}; \\
2 & \text{if } \beta\in\{(0,n),(0,i,j),(0,j,i),(1,n-1), (1,i-1,j),(1,j-1,i),(i,j),(j,i)\}.
\end{cases} \]
Thus
\[ r^D(T)=
\begin{cases}
1-2\cdot4+2\cdot4=1, & \text{if } |T|=4;\\
(-1)(1-2\cdot3+2\cdot2)=1, & \text{if } |T|=3;\\
1-2-2+2=-1, & \text{if } T\in\{\{0,i\},\{0,j\},\{1,i\},\{1,j\}\}; \\
1-2-2=-3, & \text{if } T\in\{\{0,1\},\{i,j\}\}; \\
(-1)(1-2)=1, & \text{if } |T|=1; \\
1, & \text{if } |T|=0; \\
\end{cases} \]
It follows that $c^D_{p,\pm1}(n) = 7\cdot 2^{n-4}$, $c^D_{p,\pm3}(n) = 2^{n-4}$, and $c^D_{p,i}(n)=0$ for all $i\not\equiv \pm1, \pm3\pmod p$ when $p>3$ and $c^D_3(n)=(2^{n-3},7\cdot 2^{n-4}, 7\cdot 2^{n-4})$.
\end{proof}

For small values of $p$ and $n$, we compute $c^D_p(n)$ in Sage based on its definition and provide our data in Table~\ref{tab:D}, which agrees with the results in this section; note that the power of $2$ given by Corollary~\ref{cor:D} may or may not be the highest in $c^D_{p,i}(n)$. 
\begin{table}[h]
\tiny
\renewcommand{\arraystretch}{1.2}
\begin{center}
\begin{tabular}{c|ccccc} 
\toprule
$n$ & $p=3$ & $p=5$ & $p=7$ & $p=11$ \\ 
\midrule
$4$ & $(0,8,8)$ & $(0,2,12,2,0)$ & $(6,2,2,6,0,0,0)$ & $(0,4,0,0,0,0,6,6,0,0,0)$ \\ 
$5$ & $(12,16,4)$ & $(16,8,0,0,8)$ & $(6,6,10,2,4,0,4)$ & $(6,2,2,0,4,2,6,0,2,4,4)$ \\ 
$6$ & $(16,24,24)$ & $(0,32,0,0,32)$ & $(4,12,10,8,12,6,12)$ & $(12,6,0,0,6,6,4,4,14,8,4)$ \\ 
$7$ & $(56,36,36)$ & $(16,32,24,24,32)$ & $(64,32,0,0,0,0,32)$ & $(18,4,8,4,22,16,12,10,16,16,2)$ \\ 
$8$ & $(96, 80, 80)$ & $(52, 62, 40, 40, 62)$ & $(0, 128, 0, 0, 0, 0, 128)$ & $(18, 28, 24, 18, 22, 18, 30, 24, 26, 22, 26)$ \\
$9$ & $(256, 128, 128)$ & $(104, 112, 100, 96, 100)$ & $(0, 128, 32, 96, 96, 32, 128)$ & $(36, 40, 72, 52, 40, 24, 38, 62, 62, 50, 36)$ \\
$10$ & $2^9(0, 1, 1)$ & $2^7(0, 3, 1, 1, 3)$ & $2^3(20, 15, 19, 20, 20, 19, 15)$ & $2(47, 48, 45, 54, 47, 42, 36, 50, 50, 45, 48)$ \\
$11$ & $2^7(6, 5, 5)$ & $2^6(2, 7, 8, 8, 7)$ & $2^3(32, 51, 29, 32, 32, 29, 51)$ & $2^9(2, 1, 0, 0, 0, 0, 0, 0, 0, 0, 1)$ \\
$12$ & $2^8(2, 7, 7)$ & $2^4(38, 47, 62, 62, 47)$ & $2(280, 250, 341, 293, 293, 341, 250)$ & $2^{11}(0, 1, 0, 0, 0, 0, 0, 0, 0, 0, 1)$ \\
$13$ & $2^6(38,45,45)$ & $2^3(190, 237, 180, 180, 237)$ & $2(614, 631, 521, 599, 601, 530, 600)$ & $2^9(0, 4, 0, 3, 0, 1, 1, 0, 3, 0, 4)$ \\
$14$ & $2^4(306, 359, 359)$ & $2(1473, 1777, 1620, 1595, 1727)$ & $2^{11}(0, 3, 0, 1, 1, 0, 3)$ & $2^7(24, 15, 6, 4, 9, 18, 18, 9, 4, 6, 15)$ \\
$15$ & $2^{11}(6,5,5)$ & $2^{11}(2,2,5,5,2)$ & $2^{10}(2,7,1,7,7,1,7)$ & $2^6(50,58,48,32,46,47,47,46,32,48,58)$ \\
$16$ & $2^5(606, 721, 721)$ & $2^{10}(14, 15, 10, 10, 15)$ & $2^8(22, 44, 39, 34, 34, 39, 44)$ & $2^5(260, 195, 257, 158, 160, 124, 124, 160, 158, 257, 195)$ \\
\bottomrule
\end{tabular}
\end{center}
\caption{$c^D_p(n)$ for small values of $p$ and $n$}
\label{tab:D}
\end{table}

\section{Concluding remarks}\label{sec:conclusion}

In this paper, we use a result of Dickson~\cite{Dickson} on the congruence of multinomial coefficients to determine how many ribbon numbers indexed by compositions of $n$ belong to each congruence class modulo $p$.
We apply our result to some special cases of $n$, that is, when $n$ takes the following values:
\[ mp^d,\ p^{d_1}+\cdots + p^{d_k},\ 2p^d+p^e \]
For other values of $n$, our result becomes tedious, and it would be nice to develop a different approach.
There might be an interpretation of our results by the representation theory of $H_n(0)$, or by certain operations on standard tableaux of ribbon shapes, or even by the flag $h$-vector of the Boolean algebra of subsets of $[n-1]$, which could lead to results for more values of $n$.

We also extend our result to type $B$ and type $D$; in particular, we show that the ribbon numbers are all odd in these two types.
For the Coxeter system of $I_2(m)$, it is routine to check that there are two descent classes of size $1$ and two of size $m-1$.
For Coxeter systems of exceptional types, we list the sizes of the descent classes below based on computations in Sage, where $r^k$ means $k$ descent classes of size $r$.
\begin{itemize}
\item
$F_4$: $1^2$, $23^4$, $73^2$, $95^4$, $97^2$, $169^2$
\item
$H_3$: $1^2$, $11^2$, $19^2$, $29^2$, $29$ 
\item
$H_4$: $1^2$, $119^2$, $599^2$, $601^2$, $719^2$, $1199^2$, $1681^2$, $2281^2$
\item
$E_6$: $1^2$, $26^4$, $71^2$, $190^4$, $215^4$, $217^2$, $334^4$, $530^4$, $647^4$, $649^2$, $719^2$, $793^4$, $838^4$, $1106^4$, $1151^2$, $1225^4$, $1414^4$, $1729^2$, $2042^4$, $2663^2$
\item
$E_7$: $1^2$, $55^2$, $125^2$, $575^2$, $701^2$, $755^2$, $1331^2$, $1891^2$, $2015^2$, $3331^4$, $3401^2$, $4031^2$, $5473^2$, $6679^2$, $6749^2$, $7309^2$, $7939^2$, $8009^2$, $8189^2$, $8749^2$, $9505^4$, $10025^2$, $10079^2$, $10655^2$, $10765^2$, $14687^2$, $15373^2$, $15553^2$, $16003^2$, $16829^2$, $18145^2$, $19459^2$, $20035^2$, $21491^2$, $21545^2$, $22301^2$, $22931^2$, $25577^2$, $26207^2$, $26209^2$, $26839^2$, $27469^2$, $28855^2$, $29539^2$, $30185^2$, $30925^2$, $34273^2$, $34903^2$, $38249^2$, $39635^2$, $41021^2$, $42391^2$, $44477^2$, $45107^2$, $49645^2$, $50455^2$, $55007^2$, $59149^2$, $62551^2$, $69875^2$, $73331^2$, $94121^2$
\end{itemize}
Reducing the above sizes modulo a prime $p$ gives the number $c_{p,i}^{(W,S)}(n)$ of descent classes of sizes congruent to $i$ modulo $p$ for all $i\in\mathbb{Z}_p$.
For small values of $p$, see Table~\ref{tab:exceptional}.

\begin{table}[h]
\begin{center}
\tiny
\renewcommand{\arraystretch}{1.2}
\begin{tabular}{c|cccccc} 
\toprule
Type & $p=2$ & $p=3$ & $p=5$ & $p=7$ & $p=11$ & $p=13$ \\
\midrule
$E_6$ & $2^5(1, 1)$ & $2^5(0, 1, 1)$ & $2(8, 7, 5, 5, 7)$ & $2^2(4, 2, 1, 2, 0, 7, 0)$ & $2(1, 4, 5, 2, 7, 1, 6, 3, 1, 2, 0)$ & $2(5, 5, 0, 2, 1, 0, 3, 3, 2, 3, 6, 1, 1)$ \\
$E_7$ & $2^7(0, 1)$ & $2^6(0, 1, 1)$ & $2^2(10, 8, 3, 3, 8)$ & $2^6(0, 1, 0, 0, 0, 0, 1)$ & $2(4, 5, 12, 3, 8, 7, 4, 8, 5, 4, 4)$ & $2(6, 7, 7, 7, 4, 3, 4, 3, 5, 1, 4, 7, 6)$ \\
$F_4$ & $2^4(0, 1)$ & $2^3(0, 1, 1)$ & $2(2, 1, 1, 3, 1)$ & $2(0, 2, 2, 1, 2, 0, 1)$ & $2(0, 3, 0, 0, 1, 0, 0, 3, 0, 1, 0)$ & $2(1, 1, 0, 0, 2, 0, 1, 0, 1, 0, 2, 0, 0)$ \\
$H_3$ & $2^3(0, 1)$ & $2^2(0, 1, 1)$ & $2^2(0, 1, 0, 0, 1)$ & $2(0, 2, 0, 0, 1, 1, 0)$ & $2(1, 1, 0, 0, 0, 0, 0, 1, 1, 0, 0)$ & $2(0, 1, 0, 1, 0, 0, 1, 0, 0, 0, 0, 1, 0)$ \\
$H_4$ & $2^4(0, 1)$ & $2^3(0, 1, 1)$ & $2^3(0, 1, 0, 0, 1)$ & $2(1, 2, 1, 0, 1, 1, 2)$ & $2(1, 1, 0, 0, 2, 1, 0, 1, 0, 2, 0)$ & $2(0, 2, 1, 2, 2, 0, 1, 0, 0, 0, 0, 0, 0)$ \\
$I_2(5)$ & $2(1, 1)$ & $2^2(0, 1, 0)$ & $2(0, 1, 0, 0, 1)$ & $2(0, 1, 0, 0, 1, 0, 0)$ & $2(0, 1, 0, 0, 1, 0, 0, 0, 0, 0, 0)$ & $2(0, 1, 0, 0, 1, 0, 0, 0, 0, 0, 0, 0, 0)$ \\
$I_2(6)$ & $2^2(0, 1)$ & $2(0, 1, 1)$ & $2(1, 1, 0, 0, 0)$ & $2(0, 1, 0, 0, 0, 1, 0)$ & $2(0, 1, 0, 0, 0, 1, 0, 0, 0, 0, 0)$ & $2(0, 1, 0, 0, 0, 1, 0, 0, 0, 0, 0, 0, 0)$
\\
$I_2(7)$ & $2(1, 1)$ & $2(1, 1, 0)$ & $2^2(0, 1, 0, 0, 0)$ & $2(0, 1, 0, 0, 0, 0, 1)$ & $2(0, 1, 0, 0, 0, 0, 1, 0, 0, 0, 0)$ & $2(0, 1, 0, 0, 0, 0, 1, 0, 0, 0, 0, 0, 0)$ \\
\bottomrule
\end{tabular}
\end{center}
\caption{$c_{p,i}^{(W,S)}(n)$ for Coxeter systems $(W,S)$ of exceptional types}
\label{tab:exceptional}
\end{table}

The source code for our computations can be found here: 
\url{https://cocalc.com/share/public_paths/8af477db157c1e30eb5c0bc2653ffb3e2440f1c8}.

An important tool used in this work is Theorem~\ref{thm:Dickson}, which was obtained by Dickson~\cite{Dickson} as a generalization of Lucas' theorem on congruence of binomial coefficients to multinomial coefficients.
Note that Davis and Webb~\cite{DavisWebb} generalized Lucas' theorem to prime powers. 
Thus it might be possible to generalize our results to prime powers.

Finally, while our results are all deterministic, it would be meaningful to explore probabilistic features of the ribbon numbers modulo a prime.



\begin{thebibliography}{99}


\bibitem{AG}
P. Amrutha and T. Geetha, On the Degrees of Representations of Groups not divisible by $2^k$, unpublished preprint, 2022.

\bibitem{Andrews}
G.~E. Andrews, The theory of compositions. II. Simon Newcomb's problem, Utilitas Math. {\bf 7} (1975), 33--54.

\bibitem{APS}
A. Ayyer, A. Prasad and S. Spallone, Odd partitions in Young's lattice, S\'em. Lothar. Combin. {\bf 75} ([2015--2019]), Art. B75g, 13 pp.

\bibitem{BjornerBrenti}
A. Bj\"orner and F. Brenti, {\it Combinatorics of Coxeter groups}, Graduate Texts in Mathematics, 231, Springer, New York, 2005

\bibitem{DavisWebb}
K.~S. Davis and W.~A. Webb, Lucas' theorem for prime powers, European J. Combin. {\bf 11} (1990), no.~3, 229--233.

\bibitem{Dickson}
L.~E. Dickson, The analytic representation of substitutions on a power of a prime number of letters with a discussion of the linear group, Ann. of Math. {\bf 11} (1896/97), no.~1-6, 65--120


\bibitem{GKNT}
E. Giannelli, A. Kleshchev, G. Navarro, and P. H. Tiep, Restriction of odd degree characters and natural correspondences, Int. Math. Res. Not. IMRN {\bf 2017}, no.~20, 6089--6118.

\bibitem{DivChar}
J. Ganguly, A. Prasad and S. Spallone, On the divisibility of character values of the symmetric group, Electron. J. Combin. {\bf 27} (2020), no.~2, Paper No. 2.1, 6 pp.


\bibitem{H0Tab}
J. Huang, A tableau approach to the representation theory of 0-Hecke algebras, Ann. Comb. {\bf 20} (2016), no.~4, 831--868.

\bibitem{Khanna}
A. Khanna, Enumeration of Partitions modulo $4$, arXiv:2207.07513.

\bibitem{KrobThibon}
D. Krob and J.-Y. Thibon, Noncommutative symmetric functions. IV. Quantum linear groups and Hecke algebras at $q=0$, J. Algebraic Combin. {\bf 6} (1997), no.~4, 339--376.

\bibitem{Macdonald}
I.~G. Macdonald, On the degrees of the irreducible representations of symmetric groups, Bull. London Math. Soc. {\bf 3} (1971), 189--192.

\bibitem{Miller}
A.~R. Miller, On parity and characters of symmetric groups, J. Combin. Theory Ser. A {\bf 162} (2019), 231--240.

\bibitem{Norton}
P.~N. Norton, $0$-Hecke algebras, J. Austral. Math. Soc. Ser. A {\bf 27} (1979), no.~3, 337--357.

\bibitem{EvenChar}
S. Peluse, On even entries in the character table of the symmetric group, arXiv: 2007.06652.

\bibitem{CharModPrime}
S. Peluse and K. Soundararajan, Almost all entries in the character table of the symmetric group are multiples of any given prime, J. Reine Angew. Math. {\bf 786} (2022), 45--53.

\bibitem{Solomon}
L. Solomon, A Mackey formula in the group ring of a Coxeter group, J. Algebra {\bf 41} (1976), no.~2, 255--264.

\end{thebibliography}
\end{document}